\documentclass[12pt]{amsart}
\usepackage[english]{babel}
\usepackage{amsmath}
\usepackage{amsthm}
\usepackage{amssymb}
\usepackage{mathrsfs}
\usepackage{enumerate}
\usepackage[notcite, final, notref]{showkeys}
\usepackage{dsfont}
\usepackage{booktabs}
\usepackage{graphicx}
\usepackage[usenames,dvipsnames]{color}
\usepackage{tikz}
\usetikzlibrary{arrows,calc,chains,shapes}

\tikzstyle{block1} = [draw,rectangle,thick,minimum height=1.0in,minimum width=1in]
\tikzstyle{block2} = [draw,rectangle,thick,minimum height=0.5in,minimum width=1in]
\tikzstyle{sum} = [draw,circle,inner sep=0mm,minimum size=2mm]
\tikzstyle{connector} = [->,thick]
\tikzstyle{line} = [thick]
\tikzstyle{branch} = [circle,inner sep=0pt,minimum size=1mm,fill=black,draw=black]
\tikzstyle{guide} = []

\newcommand{\xdownarrow}[1]{
  {\left\downarrow\vbox to #1{}\right.\kern-\nulldelimiterspace}}
\newcommand{\xuparrow}[1]{%
  {\left\downarrow\vbox to #1{}\right.\kern-\nulldelimiterspace}
}

\newtheorem{theorem}{Theorem}[section]

\newtheorem{lemma}[theorem]{Lemma}
\newtheorem{proposition}[theorem]{Proposition}
\newtheorem{corol}[theorem]{Corollary}
\newtheorem{definition}[theorem]{Definition}
\newtheorem{axioms}[theorem]{Axioms}

\theoremstyle{definition}
\newtheorem{remark}[theorem]{Remark}
\newtheorem{example}[theorem]{Example}

\newcommand{\w}{\omega}

\usepackage{xcolor}

\title[$W$-Markov measures, transfer operators and wavelets]{$W$-Markov measures, transfer operators, wavelets 
and multiresolutions}

\author[D. Alpay]{Daniel Alpay}
\address{(DA) Department of Mathematics\\
Ben-Gurion University of the Negev\\
Beer-Sheva 84105 Israel}
\email{dany@math.bgu.ac.il}

\author[P. Jorgensen]{Palle Jorgensen}
\address{(PJ)  Department of Mathematics\\
 University of Iowa.
Iowa City, IA 52242 USA}
\email{palle-jorgensen@uiowa.edu}

\author[I. Lewkowicz]{Izchak Lewkowicz}
\address{(IL) Department of Electrical \& Computer Engineering\\
Ben-Gurion University of the Negev\\
Beer-Sheva 84105 Israel} \email{izchak@ee.bgu.ac.il}

\begin{document}
\begin{abstract}
In a general setting we solve the following inverse problem: Given a positive operators $R$, acting on measurable functions on a 
fixed measure space $(X,\mathcal  B_X)$, we construct an associated Markov chain. Specifically, starting with a choice of $R$ 
(the transfer operator), and a probability measure $\mu_0$ on $(X, \mathcal B_X)$, we then build an associated Markov chain 
$T_0, T_1, T_2,\ldots$, with these random variables (r.v) realized in a suitable probability space $(\Omega,\mathcal F, \mathbb P)$, 
and each r.v.  taking values in $X$, and with $T_0$ having the probability $\mu_0$ as law. We further show how spectral data for $R$, e.g., 
the presence of $R$-harmonic functions,  propagate to the Markov chain.  Conversely, in a general setting, we show that every Markov chain 
is determined by its transfer operator.
In a range of examples we put this correspondence into practical terms:  $(i)$ iterated function systems (IFS), $(ii)$ 
wavelet multiresolution constructions, and $(iii)$  IFSs with random “control.”
Our setting for IFSs is general as well: a fixed measure space $(X, \mathcal B_X)$ and a system of mappings $\tau_i$, each acting in 
$(X, \mathcal B_X)$, and each assigned a probability, say $p_i$   which may or may not be a function of $x$. For standard IFSs, the $p_i$'s 
are constant, but for wavelet constructions, we have functions $p_i(x)$ reflecting the multi-band filters which make up the wavelet 
algorithm at hand. The sets $\tau_i(X)$  partition $X$, but they may have overlap, or not.
  For IFSs with random control, we show how the setting of transfer operators translates into explicit Markov moves:  Starting with a point 
$x\in X$, the Markov move to the next point is in two steps, combined yielding the move from $T_0 = x$ to $T_1 = y$, and more generally from 
$T_n$ to $T_{n+1}$. The initial point $x$ will first move to one of the sets $\tau_i(X)$  with probability  $p_i$, and once there, 
it will “choose” a definite position $y$ (within $\tau_i(X)$), now governed by a fixed law (a given probability distribution). 
For Markov chains, the law is the same in each move from $T_n$ to $T_{n+1}$.
  \end{abstract}


\thanks{{\sl Acknowledgments:} D. Alpay thanks the Earl Katz family for
endowing the chair which supported his research. The second named author (PJ) wishes to thank the department of mathematics at
Ben-Gurion University for hospitality during a 4 weeks research visit in the Spring of 2016, allowing for collaboration, and completion of this research.}

\maketitle

\date{today}
\tableofcontents
\section{Introduction}
\setcounter{equation}{0}
The purpose of our paper is to explore in two directions the interconnection between positive operators $R$ defined in certain function spaces, on the one hand, and associated discrete time-random processes on the other. The direction back from $R$ to the discrete time-random process, we refer to as “the inverse problem.” It includes the construction of the process itself. By contrast, the direct problem starts with a given discrete time-random process, and then computes the associated transfer operator, or sequence of transfer operators, and then finally uses 
the latter in order to determine properties of the given random process under consideration.\smallskip

Our second purpose is a list of applications of our results in the general setting, the applications ranging from homogeneous Markov chains with white noise-input, dynamics of endomorphisms, including logistics maps, encoding mappings, invariant measures, wavelets in a general setting of multi-resolutions and associated transfer operators, also called Ruelle operators. In the case of a single 
positive operators $R$, we obtain, via a solution to the inverse problem, an associated generalized Markov processes, but its detailed properties will depend on a prescribed weight function $W$, hence the term “$W$-Markov processes.” In the case of a prescribed sequence of positive operators, we still obtain associated discrete time-random processes, now with each operator $R_n$ accounting for the transfer of information from time $n$ to time $n+1$. But these processes will not be Markov. Hence the Markov property is equivalent to $R_n = R$ for all $n$.\smallskip

Returning to the case of our study of dynamics of endomorphisms, 
say $\sigma$ in $X$, if the transfer operator $R$ 
is $\sigma$-homogeneous, we show that the associated Markov processes will be of a special kind: when realized in the natural probability space of an associated solenoid ${\rm Sol}_\sigma(X)$ (see Definition \ref{sol} for the latter), we arrive at multi-scale resolutions in 
$\mathbf L_2({\rm Sol}_\sigma(X), \mathcal F,\mathbb P)$ (see Definition \ref{newdef123}), with the scale of resolutions in question defined from the given endomorphism 
$\sigma$. In the case when $\sigma$ is the scale endomorphism of a wavelet construction, we show that the wavelet multi-scale resolution will agree with that of the associated solenoid analysis. The latter framework is much more general, and covers a variety of multiresolution
models.\smallskip

\begin{table}[h!]
\caption{Increasing level of generality (each with its transfer operator and multiresolution; see Tables \ref{table1234} and \ref{rubicon2})}
\begin{tabular}{|c|c|c|c|}
\toprule
& &&\\
\text{Case}&$\mathbf L_2(\mathbb R,dx)$& $\mathbf L_2({\rm Sol}_\sigma(X),\mathbb P)$& $\mathbf L_2(\Omega,\mathbb P)$\\
&$\longrightarrow$&$\longrightarrow$&\\
\bottomrule
\end{tabular}
\label{table1}
\end{table}

Before turning to the third theme in our paper, a few words on terminology: by a measure space $(X, \mathcal B_X)$ we mean a set $X$ and a 
sigma-algebra $\mathcal B_X$ of subsets, each specified at the outset, usually with some additional technical restrictions. 
By a probability space, we mean a triple $(\Omega, \mathcal F, \mathbb P)$, sample space $\Omega$, sigma-algebra of events $\mathcal F$, and 
probability measure $\mathbb P$. We shall consider systems of random variables with values in measure spaces $(X, \mathcal B_X)$; different 
random variables may take values in different measure spaces. Our first order of business is to show that for any pair of random variables, 
say $A$ and $B$, each taking values in a measure space, there is an associated transfer operator $R$, depending only on $A$ and $B$, 
which “transfers”” information from one to the other. If $A$ and $B$ are independent, the associated operator $R$ will be of rank-one, while 
if the sigma algebra generated by $A$ is contained in that of $B$, then $R$ will be the inclusion operator of the $L_2$-spaces of the 
respective distributions, the distribution of $A$ and that of $B$.\smallskip

One source of motivation for our present work is a number of recent papers dealing with generalized wavelet multiresolutions, see e.g., 
\cite{MR1855241,MR2371594,MR2391805,MR2970658, MR2629692, MR2457327,MR2018241}, and harmonic analysis on groupoids. While these themes may seem disparate, they are connected via a set of questions in operator algebra theory; see e.g., 
\cite{MR2945156,MR2821778,MR2966144}. The positive operators considered here are in a general measure theoretic setting, but we stress that there is also a rich theory of positive integral operators is the metric space setting, often called Mercer operators, and important in the approach of Smale and collaborators to learning theory, see e.g., \cite{MR1864085,MR2810909,MR2488871}. However for our present use, the setting of the Mercer operators is too restrictive.
\smallskip

While various aspects of our settings may have appeared in special cases in anyone or the other of existing treatments of Markov chains, 
the level of generality, the questions addressed, and the specific and detailed interconnections, some surprising, revealed below, 
we believe have not. Relevant references include \cite{MR1689633,MR544839,MR1974383} and the papers cited therein.\smallskip

Aside from the Introduction, the paper is divided into three sections. Since our approach to the applications involves some issues of a 
general nature, we found it best to begin with general theory, Section 2, covering a number of new results, all based on several 
intriguing operator theoretic features of  general systems of random variables, and their associated transfer operators. This is developed 
first, and its relevance to discrete-time random processes is then covered in the remaining of Section 2.
From there, we then turn to Markov chains, developed in this rather general and operator theoretic framework, and with an emphasis on 
transfer operator related issues. It is our hope that this will be of interest to readers both in operator theory, and in random dynamical 
systems and their harmonic analysis. We have thus postponed the applications to the last section. This is dictated in part by  our focus on 
those Markov chains and associated dynamical systems which are induced by endomorphisms in measure spaces. In Section 3 we show that this
setting can  be realized in probability spaces over solenoids. Each endomorphism induces a solenoid, and a Markov chain of a special kind. 
The usefulness of this point of view is then documented with a host of applications and detailed examples which we have 
included in several subsections in Section 4.
We believe that our results in both the general theory and in our applications sections are of independent interest.

\section{General theory}
\setcounter{equation}{0}
\label{sec-2}
In this section, we consider the following general setting of random variables systems (r.v.s)  on a prescribed probability space  
$(\Omega, \mathcal F, \mathbb P)$, each r.v. taking values in a measure space $(X, \mathcal B_X)$; different random variables may 
take values in different measure spaces. Our aim is to make precise transfer between the different r.v.s making up the system. For this purpose we concentrate on the case of a pair of r.v.s, say $A$ and $B$. There is then an associated transfer operator $R = R_{A,B}$, depending only on $A$ and $B$, which “transfers”” information from one to the other. The transfer operator makes precise the “intertwining” of the two random variables. Indeed, if $A$ and $B$ are in fact given to be independent, then the associated operator $R$ will be of rank-one, or zero in the case of zero means. On the other hand, if the sigma algebra generated by $A$ is contained in that of $B$, then $R$ will be the inclusion operator of the $L_2$-spaces of the respective distributions, i.e., the distribution of $A$ and that of $B$. We further show, in the general setting, that the product of the respective conditional expectations (the one for $A$ and the one for $B$) are linked, via a factorization formula,  by the transfer operator $R_{A,B}$. See Table \ref{rubicon} below.\\

While Section \ref{sec-2} is somewhat long and technical, it serves two important purposes: one, it offers lemmas to be used in the proofs of 
our main theorems later. The second purpose is to develop the tools we need in several inductive limit constructions to be used in our 
analysis of inverse problems, the inductive limits here concern the step of realizing infinite-dimensional discrete time-random processes as 
inductive limits of finite systems. For the finite systems themselves we develop here (the first five lemmas in Section \ref{sec-2}) 
a new kernel analysis which will then be used later when we build the infinite dimensional probability models needed in the main theorems.
As mentioned, a key tool is the notion of a transfer operator for a pair (or a finite number of) random variables. We shall include an 
analysis of the special case when one of the two r.v.s takes values in a discrete measure space. There are two reasons for this, one the 
interest in Markov chains with discrete state space, and the other is the study of such random variables as stopping time 
(see Definition \ref{stoppingtime}).\smallskip

Our approach to the analysis of finite systems of r.v.s is operator theoretic, relying on systems of isometries, co-isometries and projections, the latter in the form of conditional expectations. Of independent interest is our Corollary \ref{corol224} which offers a representation 
of some operator relations known as the Cuntz-Krieger relations in operator algebra theory. Lemmas \ref{lemma22}, \ref{lemma24}, 
\ref{prop29}, \ref{wzero}, and \ref{Wis} prepare the ground for what is to follow.
  Main results in the section includes Theorems \ref{danielle123}, \ref{michelle123}, \ref{loudmila}, \ref{thhmc}, and \ref{danielle}, 
as well as their corollaries and applications. Theorem \ref{michelle123} offers a model for the analysis of Markov processes in the general 
setting of our paper, Theorem \ref{loudmila} is a result which supplies a model for Markov chains driven by white noise. In this case we 
also compute an explicit invariant measure. This in turn is applied (Theorem \ref{danielle}) to a new random process realized naturally in a probability space over the Schur functions from complex analysis. Background references on calculus of random variables include 
\cite{MR1669737,MR1974383,Ka48,MR0203748,MR33:6659,MR1600720}; on classes of positive operators (Ruelle operators) 
\cite{MR1793194,MR3275999,xMR2599889,MR1837681}; and on algebras of operators in Hilbert space
\cite{aron,Cun77,MR592141,MR2821778,MR3441736,MR1171011,MR2457327,Nelson_flows,MR647807}.

\subsection{Pairs of random variables and transfer operators}
\label{sec2_1}
Let $(X,\mathcal B_X)$ be a measurable space. In this section, we define a transfer operator associated with two $X$-valued random variables, 
say $A$ and $B$, defined on some probability space $(\Omega,\mathcal F,\mathbb P)$. The distribution probability of $A$ (also called ``law'')
is defined by
\[
\mu_A(L)=\mathbb P(A^{-1}(L)),\quad L\in\mathcal B_X,
\]
and so, with $\mathcal M(X,\mathcal B_X)$ denoting the space of real-valued measurable functions defined on $X$,
\[
\int_\Omega f(A(\w))d\mathbb P(\w)=\int_Xf(x)d\mu_A(x),\quad \forall f\in\mathcal M(X,\mathcal B_X),
\]
(and similarly for $B$).

\begin{definition}{\rm 
We denote by $\mathcal F_A$ the sub sigma-algebra of $\mathcal F$ defined by
\begin{equation}
\label{rachel}
\mathcal F_A=\left\{A^{-1}(L)\,;\, L\in\mathcal B\right\}.
\end{equation}
}
\label{deffa}
\end{definition}

By definition of $\mu_A$, and with $\mathcal F_A$ introduced in Definition \ref{deffa}, the map 
\begin{equation}
V_Af=f\circ A
\label{vn}
\end{equation}
is an isometry from $\mathbf L_2(X,\mathcal B_X,\mu_A)$ onto $\mathbf L_2(\Omega,\mathcal F_A,\mathbb P)$. For the adjoint operator $V_A^*$ we 
have the following covariance (in a sense analogue to the one in mathematical physics and representation theory).

\begin{lemma} It holds that
\begin{equation}
(V_A^*\psi)(x)=\mathbb E_{A=x}(\psi\,|\,\mathcal F_A),\quad \psi\in\mathbf L_2(\Omega,\mathcal F,\mathbb P).
\end{equation}
\label{lemma22}
\end{lemma}
\begin{proof} We take $\psi\in\mathbf L_2(\Omega,\mathcal F,\mathbb P)$ and $f\in\mathbf L_2(X,\mu_A)$.
We have
\[
\begin{split}
\langle V_A^*\psi,f\rangle_{\mu_A}&=\langle \psi, V_Af\rangle_{\mathbb P}\\
&=\langle \psi ,f\circ A\rangle_{\mathbb P}\\
&=\int_{\Omega}\psi(\w)f(A(\w))d\mathbb P(\w)\\
&=\int_{\Omega}f(A(\w))\mathbb E\left(\psi\,|\, \mathcal F_A\right)d\mathbb P(\w).
\end{split}
\]
But $\mathcal F_A$ is generated by the functions of the form
\[
\chi_{A^{-1}(\Delta)}=\chi_\Delta\circ A,\quad \Delta\in\mathcal B_X,
\]
and so there is a uniquely determined function $g\in\mathcal M(X,\mathcal B_X)$ such that $E\left(\psi\,|\, \mathcal F_A\right)=g\circ A$. 
(Uniqueness of $g$ follows from the fact that $V_A\,:\,\mathbf L_2(X,\mathcal B_X,\mu_A)\,\longrightarrow\,\mathbf L_2(\Omega,\mathcal F_A,\mathbb P)$ is an 
isometry). Hence
\[
\begin{split}
\langle V_A^*\psi,f\rangle_{\mu_A}&=\int_{\Omega}f(x)g(x)d\mu_A(x),
\end{split}
\]
and hence the formula,
\[
(V_A^*\psi)(x)=g(x)=\mathbb E_{A=x}\left(\psi\,|\,\mathcal F_A\right).
\]
\end{proof}

\begin{corol}
The measure $(\psi d\mathbb P)\circ A^{-1}$ is absolutely continuous with respect to $\mu_A$, and for
$\psi\in\mathbf{L}_2(\Omega, \mathcal{F}, \mathbb{P})$
we have
\begin{equation}
\left(\psi d\mathbb P\right)\circ A^{-1}=gd\mu_A,
\end{equation}
and
\begin{equation}
V_A^*\psi=\frac{\left(\psi d\mathbb P\right)\circ A^{-1}}{d\mu_A}.
\end{equation}
\end{corol}
\begin{proof}
From the previous proof we have on the one hand
\[
\begin{split}
\langle V_A^*\psi,f\rangle_{\mu_A}&=\int_{\Omega}f(A(\w))\left(\psi(\w)d\mathbb P(\w)\right)\\
&=\int_Xf(x)\left(\left(\psi d\mathbb P\right)\circ A^{-1}\right)(x)
\end{split}
\]
and on the other hand,
\[
\begin{split}
\langle V_A^*\psi,f\rangle_{\mu_A}&=\int_{\Omega}f(A(\w))\left(\psi(\w)d\mathbb P(\w)\right)\\
&=\int_{\Omega}f(x)g(x)d\mu_A(x)\\
&=\int_{\Omega}f(A(\w))g(A(\w))d\mathbb P(\w),
\end{split}
\]
and the claim follows by comparing these two computations.
\end{proof}
With the above random variables $A, B$, we associate the positive operator $R_{A, B}$, which we call the transfer operator from $A$ to $B$, 
defined by
\begin{equation}
\label{lycee_llg}
R_{A,B}=V_A^*V_B,
\end{equation}
see the figure below:

\[
\begin{array}{ccc}
\mathbf L_2(\mu_B)&\xrightarrow{\hspace*{0.5cm}{R_{A,B}}\hspace*{.5cm}}
%
&\mathbf L_2(\mu_A)\\ \stackrel{V_B}{\searrow}
& &\stackrel{V_A^*}{\nearrow}\\
&\mathbf L_2(\Omega,\mathbb P)&
\end{array}.\]

Note that both $V_A^*$ and $R_{A,B}$ are positive operators in the following sense:
\[
\psi\ge\,0\,\,\Longrightarrow\,\, V_A^*\psi\ge\,0
\]
and
\[
f\ge 0\,\,\Longrightarrow\,\, V_A^*V_Bf\ge 0.
\]

The following result shows that $R_{A,B}$ is a conditional expectation. In \eqref{elisabeth}, by $\mathbb E(\cdot\,\big|\,\mathcal F_A)$ we mean the 
orthogonal projection of $\mathbf L_2(\Omega,\mathcal F,\mathbb P)$ onto $\mathbf L_2(\Omega,\mathcal F_A,\mathbb P)$. It can also be defined as
\begin{equation}
\mathbb E\left(\psi\,\big|\,\mathcal F_A\right)=\frac{d(\psi d\mathbb P)}{d\mathbb P_{\mathcal A}}
\end{equation}
in terms of Radon-Nikodym derivatives; see \cite{Nel73}.\smallskip

\begin{lemma}
We have:
\begin{eqnarray}
\label{elisabeth}
\mathbb E\left(f\circ A\,\big|\,\mathcal F_B\right)&=&(R_{A,B}^*f)\circ B\,=\,(R_{B,A}f)\circ B,\,\quad f\in\mathbf L_2(X,\mathcal B_X,\mu_A)\\
\mathbb E\left(g\circ B\,\big|\,\mathcal F_A\right)&=&(R_{A,B}g)\circ A\,=\,(R_{B,A}^*g)\circ A,\,\quad\,\, g\in\mathbf L_2(X,\mathcal B_X,\mu_B).
\label{murielle123}
\end{eqnarray}
\label{lemma24}
\end{lemma}
\begin{proof} We prove \eqref{murielle123}. The proof of \eqref{elisabeth} is similar and follows from $(V_A^*V_B)^*=V_B^*V_A$.
Let $f_1\in\mathbf L_2(X,\mathcal B_X,\mu_B)$, and $f_2\in\mathbf L_2(X,\mathcal B_X,\mu_A)$. On the one hand, we have
\[
\langle V_A^*V_Bf_1,f_2\rangle_{\mu_A}=\int_\Omega((R_{A,B}f_1)\circ A)(\w)(f_2\circ A)(\w)d\mathbb P(\w).
\]
On the other hand,
\[
\begin{split}
\langle V_A^*V_Bf_1,f_2\rangle_{\mu_A}&=
\langle V_Bf_1,V_Af_2\rangle_{\mathbb P}\\
&=\int_{\Omega}(f_1\circ B)(\w)(f_2\circ A)(\w)d\mathbb P(\w)\\
&=\int_{\Omega}\left(\mathbb E\left(f_1\circ B)\,\big|\,\mathcal F_A\right)\right)(\w)(f_2\circ A)(\w)d\mathbb P(\w)
\end{split}
\]
by definition of the conditional expectation, and the result follows.
\end{proof}

\begin{corol}
Let $A, B$ and $C$ be three random variables with transfer functions
\[
R_{A,B}\,:\,\mathbf L_2(X,\mathcal B_X,\mu_B)\,\longrightarrow\,\mathbf L_2(X,\mathcal B_X,\mu_A)
\]
and
\[
R_{B,C}\,:\,\mathbf L_2(X,\mathcal B_X,\mu_C)\,\longrightarrow\,\mathbf L_2(X,\mathcal B_X,\mu_B).
\]
Then the following chain rule holds for all $f\in\mathbf L_2(X,\mathcal B_X,\mu_C)$ and $x\in X$:
\begin{equation}
\left(R_{A,B}R_{B,C}f\right)(x)=\mathbb E_{A=x}\left((R_{B,C}(f))\circ B\,\big|\, \mathcal F_A\right).
\end{equation}
\label{corol25}
\end{corol}

\begin{proof}
We have
\[
\begin{split}
\left(R_{A,B}R_{B,C}f\right)(x)&=\left(V_A^*V_BV_B^*V_Cf\right)(x)\\
&=\left(V_A^*\mathbb E_BV_Cf\right)(x)\\
&=\left(V_A^*\mathbb E_B\left(f\circ C\,\big|\,\mathcal F_B\right)\right)(x)\\
&=\left(V_A^*\left(\left(R_{B,C}(f)\right)\circ B\right)\right)(x),
\end{split}
\]
and the result follows from Lemma \ref{lemma22}.
\end{proof}

In the following lemma, $X$ is assumed locally compact, and $C_c(X)$ denotes the space of continuous functions on $X$ with compact support.

\begin{lemma}
Assume that $X$ is a locally compact topological space, and that $\mathcal B$ is the associated Borel sigma-algebra. Assume moreover that
$R_{A,B}$ sends $C_c(X)$ into $C(X)$. Then it holds that
\label{lemma4}
\begin{equation}
\label{christine}
\mathbb E\left(f\circ B\,\big|\, A=x\right)=(R_{A,B}(f))(x).
\end{equation}
\label{lemma26}
\end{lemma}
\begin{proof} We denote by $\mathcal F_{A,x}$ the sigma-algebra generated by the set $\left\{A=x\right\}$. 
We have
\[
\mathbb E(f\circ B\,\big|\, \mathcal F_{A,x})=\mathbb E\left(f\circ B\,\big|\,\mathcal F_A\right)\,\big|\, \mathcal F_{A,x}).
\]
Using the previous lemma, we can then write
\[
\begin{split}
\int_{A=x}(R_{A,B}f\circ A)(\w)d\mathbb P(\w)&=\int_{A=x}\mathbb E(f\circ B\,\big|\, \mathcal F_{A,x})d\mathbb P(\w),\quad{\rm and}\\
\int_{A\not =x}(R_{A,B}f\circ A)(\w)d\mathbb P(\w)&=\int_{A\not=x}\mathbb E(f\circ B\,\big|\, \mathcal F_{A,x})d\mathbb P(\w),
\end{split}
\]
from which we get \eqref{christine}.
\end{proof}

\begin{theorem}
Let the following be as above: The probability space $(\Omega,\mathcal F,\mathbb P)$, the random
variables $A$ and $B$, and the respective measures $\mu_A$ and $\mu_B)$. Let also $R_{A,B}$ be the corresponding transfer operator. Then the following
are equivalent:\\
$(i)$ $\mu_B<<\mu_A$, with $W=\frac{d\mu_B}{d\mu_A}$\smallskip

$(ii)$ It holds that
\begin{equation}
\label{lab09}
\int_XR_{A,B}(f)(x)d\mu_A(x)=\int_Xf(x)W(x)d\mu_A(x),
\end{equation}
that is, $\frac{d\mu_AR_{A,B}}{d\mu_A}=W$.
\label{danielle123}
\end{theorem}

\begin{proof} 
Let $f_1\in\mathbf L_2(X,\mathcal B_X,\mu_A)$ and $f_2\in\mathbf L_2(X,\mathcal B_X,\mu_B)$. From the proof of Lemma \ref{lemma4} 
we have:
\[
\begin{split}
\int_X f_1(x)\left(R_{A,B}f_2\right)(x)d\mu_A(x)&=\int_\Omega (f_1\circ A)(\w)(f_2\circ B)(\w)d\mathbb P(\w)\\
\end{split}
\]
Setting $f_1(x)\equiv 1$, we obtain
\[
\begin{split}
\int_X \left(R_{A,B}f_2\right)(x)d\mu_A(x)&=\int_X f_2(x)d\mu_B(x),
\end{split}
\]
so that $d(\mu_AR_{A,B})=d\mu_B$. By definition of $W$, we obtain \eqref{lab09}. The converse is clear.
\end{proof}

One can associate with the transfer operator
$R_{A, B}$ two extreme cases:
On the one end, if ${\rm rank}\,R_{A, B}=1$,
this corresponds to having $A$ and $B$ independent,
see Proposition \ref{prop1}. No information is passed from $A$ to $B$. 
On the other end, if $R_{A, B}=I$, it corresponds to the sub-Markovian case.

\begin{proposition}
\label{prop1}
The random variables $A$ and $B$ are independent if and only if the transfer operator $R_{A,B}$ has rank $1$.
\end{proposition}

\begin{proof}
Indeed, assume first $A$ and $B$ independent, and let $f\in\mathbf L_2(X,\mathcal B_X,\mu_A)$ and $g\in\mathbf L_2(X,\mathcal B_X,\mu_B)$. 
We have:
\[
\begin{split}
\langle f,R_{A,B}g\rangle_{\mu_A}&=\langle V_Af,V_Bg\rangle_{\mathbb P}\\
&=\int_\Omega ((f\circ A)(\w))((g\circ B)(\w))d\mathbb P(\w)\\
&=\left(\int_\Omega (f\circ A)(\w)\right)\left(\int_\Omega(g\circ B)(\w)d\mathbb P(\w)\right)\\
&=\left(\int_Xf(x)d\mu_A(x)\right)\left(\int_Xg(x)d\mu_B(x)\right)\\
&=\langle f,1\rangle_{\mu_A}\langle 1,g\rangle_{\mu_B},
\end{split}
\]
the product of the means of the respective random variables $f(A)$ and $g(B)$,
and hence $R_{A,B}$ has rank one. In Dirac's notation, (ket-bra) we can write
\[
R_{A,B}=|1>_{\mu_A}<1|_{\mu_B}.
\]

\end{proof}

Given two projections $P_1$ and $P_2$ on a Hilbert space, we recall (see \cite[p. 376]{aron} that the sequence $(P_2P_1)^m$ converges
strongly to the projection on the intersection of the corresponding spaces. Applied to $P_1=\mathbb E_B$ and $P_2=\mathbb E_A$ we obtain
that $\lim_{m\rightarrow\infty}\left(\mathbb E_A\mathbb E_B\right)^m$ is the projection onto 
$\mathbb E_A(\mathbf L_2(\Omega,\mathcal F,\mathbb P))\cap \mathbb E_B(\mathbf L_2(\Omega,\mathcal F,\mathbb P))$, that is the orthogonal
projection onto $\mathbf L_2(\Omega,\mathcal F_A\cap\mathcal F_B,\mathbb P)$.\smallskip

Here we have a more precise formula:

\begin{lemma}
With $A,B$ and $\mathbb E_A,\mathbb E_B$ as above, let $P$ denote the orthogonal projection onto the eigenspace corresponding to the 
eigenvalue $1$ of $\lim_{m\rightarrow\infty}\left(\mathbb E_A\mathbb E_B\right)^m$. Then,
\begin{equation}
\lim_{m\rightarrow\infty}\left(\mathbb E_A\mathbb E_B\right)^m\psi=
\mathbb E\left(\psi\,|\,\mathcal F_A\cap\mathcal F_B\right)=V_APR_{A,B}V_B^*\psi,\quad \forall
\psi\in\mathbf L_2(\Omega,\mathcal F,\mathbb P).
\label{juliette1234}
\end{equation}
\label{prop29}
\end{lemma}

\begin{proof}
The proof follows from the formula
\begin{equation}
\label{sylvie123}
\left(\mathbb E_A\mathbb E_B\right)^{m+1}=V_A(R_{A,B}R_{A,B}^*)^mR_{A,B}V_B^*,\quad m=0,1,\ldots
\end{equation}
which is true for $m=0$ and proved by induction as follows:
\[
\begin{split}
\left(\mathbb E_A\mathbb E_B\right)^{m+1}&=\left(\mathbb E_A\mathbb E_B\right)^{m}V_AR_{A,B}V_B^*\\
&=\overbrace{\left(V_A(R_{A,B}R_{A,B}^*)^{m-1}R_{A,B}V_B^*\right)}^{\text{induction at rank $m$}}\overbrace{\left(V_AR_{A,B}V_B^*
\right)}^{\mathbb E_A\mathbb E_B}\\
&=V_A(R_{A,B}R_{A,B}^*)^{m-1}R_{A,B}\underbrace{V_B^*V_A}_{R_{A,B}^*}R_{A,B}V_B^*\\
&=V_A(R_{A,B}R_{A,B}^*)^{m-1}R_{A,B}R_{A,B}^*R_{A,B}V_B^*\\
&=V_A(R_{A,B}R_{A,B}^*)^mR_{A,B}V_B^*.
\end{split}
\]
To conclude  we remark that $\lim_{m\rightarrow\infty}\left(\mathbb E_A\mathbb E_B\right)^m$, being a projection,
has spectrum consisting of the eigenvalues $0$ and $1$. Indeed, let $S=R_{A,B}R_{A,B}^*$. By the assumptions, the projection-valued spectral
resolution $E^{(S)}$ of the self-adjoint operator $S$ satisfies
\[
S=\int_0^1tE^{(S)}(dt),
\]
and so $\lim_{m\rightarrow\infty}S^m=E^{(S)}\left(\left\{1\right\}\right)$, where $E^{(S)}\left(\left\{1\right\}\right)$ 
(denoted by $P$ in \eqref{juliette1234}) is the spectral projection
onto 
\[
\left\{f\,\in\mathbf L_2(X,\mathcal B_X,\mu_A)\,:\, Sf=f\right\}.
\]  
As a result we get
\[
\mathbb E\left(\cdot\,\big|\,\mathcal F_A\cap \mathcal F_B\right)=V_AE^{(1)}R_{A,B}V_B^*.
\]
\end{proof}

For a related result, see \cite{MR647807}.\smallskip

As a corollary we have (where here and in the sequel we denote by $\mathbb E_A$ the conditional expectation onto $\mathcal F_A$):

\begin{corol}
In the notation of the previous proposition and of its proof, let $S=R_{A,B}R_{A,B}^*$, and let $f\in\mathbf L_2(X,\mathcal B_X,\mu_A)$ and
$\psi=V_Af$. The following are equivalent:\\
$(1)$ $Sf=f$, i.e., $E^{(S)}\left(\left\{1\right\}\right)f=f$.\smallskip

$(2)$  $\psi$ satisfies $\mathbb E_A\mathbb E_B\psi=\psi$.\smallskip

$(3)$ $\psi$ satisfies $\mathbb E_B\mathbb E_A\psi=\psi$.\smallskip

$(4)$ $\mathbb E\left(\psi\,\big|\,\mathcal F_A\cap\mathcal F_B\right)=\psi$.
\end{corol}

\begin{proof} If $T$ is a contraction from a Hilbert space $\mathcal H$ into itself  and $T\psi=\psi$ for
some $\psi\in\mathcal H$, then we also have $T^*\psi=\psi$. Indeed, using $T\psi=\psi$ we obtain
\[
\|\psi-T^*\psi\|^2=\|T^*\psi\|^2-\|\psi\|^2,
\]
which is negative since $T^*$ is also a contraction. Hence $\|\psi-T^*\psi\|=0$ and $T^*\psi=\psi$. The proof of the corollary follows then 
by applying the above fact to $T=\mathbb E_A\mathbb E_B$.
\end{proof}

\begin{corol}
In the notation of the previous proposition, the following are equivalent for pairs of random variables $A$ and $B$:\\
$(1)$ $\mathcal F_A\subset\mathcal F_B$, (that is containment of the sigma-algebras of subsets of $\Omega$)\smallskip

$(2)$ $\mathbb E_A(\mathbf L_2(\Omega,\mathbb P))\,\subset\, \mathbb E_B(\mathbf L_2(\Omega,\mathbb P))$.\\

$(3)$ $\mathbb E_A\mathbb E_B=\mathbb E_A$, or equivalently $\mathbb E_A\le\mathbb E_B$, where $\le$ denotes the standard ordering of 
projections.\\

$(4)$ $\mathbb E_B\mathbb E_A=\mathbb E_A$, equivalently $\mathbb E_A\le\mathbb E_B$.\\

$(5)$ $R_{A,B}V_B^*=V_A^*$\\

$(6)$ $V_BR_{B,A}=V_A$.
\label{corol211}
\end{corol}

\begin{proof}
This is essentially from the above, but see also the  arguments outlined in Table \ref{rubicon} below.
\end{proof}
\subsection{A formula for the conditional expectation}
\label{stonesec}
We are in the setting of Section \ref{sec2_1}. Let $A$ be a $X$-valued random variable.  For $f\in \mathcal M(X,
\mathcal B_X)$ we denote by $M_{f\circ A}$ the operator of multiplication by $f\circ A$,  from
$\mathbf L_2(\Omega,\mathcal F,\mathbb P)$ into itself. The space of all these operators when $f$ runs through
$\mathbf L^\infty(X,\mathcal B_X)$ is a commutative von Neumann algebra, denoted $\mathcal M_A$. By Stone's theorem (see
\cite{Nelson_flows}), there exists a $\mathcal M_A$-valued measure $\mathscr E_A$ on $(X,\mathcal B_X)$ such that
\begin{equation}
\label{condi123}
M_{f\circ A}=\int_Xf(x)\mathscr E_A(dx).
\end{equation}

For every $L\in\mathcal B_X$, the operator $\mathscr E(L)\in\mathcal M_A$, and so is of the form $f\circ A$ for some 
$f\in\mathbf L^\infty(X,\mathcal B_X)$, namely $f=\chi_L$.  From the equality
\begin{equation}
\label{EA}
\mathscr E_A(L)=\chi_{\left\{A\in L\right\}}=\chi_{A^{-1}(L)},\quad L\in\mathcal B_X,
\end{equation}
we shall use the notation (after identifying the function and the corresponding multiplier)
\begin{equation}
\label{othernotation}
\mathscr E_A(dx)=M_{\chi_{\left\{A\in dx\right\}}}=\chi_{\left\{A\in dx\right\}}
\end{equation}
and rewrite \eqref{condi123} as
\begin{equation}
\label{condi12345}
f\circ A=\int_Xf(x)\chi_{\left\{A\in dx\right\}},\quad{\rm or}\quad (f\circ A)(\w)=\int_Xf(x)\chi_{\left\{A(\w)\in dx\right\}}.
\end{equation}

\begin{remark}{\rm
While $\chi_{\left\{A\in dx\right\}}$ is a heuristic notation, we stress that it is made precise via the spectral theorem in the form 
\eqref{condi123}, and also by the conclusion of the next theorem.}
\end{remark}

\begin{theorem}
Let $f\in\mathbf L_2(X,\mathcal B_X),\mu_A)$. Then,
\begin{equation}
\int_Xf(x)^2d\mu_A(x)=\int_{\Omega}\left(\int_Xf(x)\chi_{\left\{A(\w)\in dx\right\}}\right)^2d\mathbb P(\w).
\end{equation}
\end{theorem}
\begin{proof}  Consider finite partitions $\pi=\left\{L_i,\,i=1,\ldots, m\right\}$ of $X$ into sets of $\mathcal B$ such that $L_i\cap L_j=0$ for $i\not=j$, and for every $i$ chose $x_i\in L_i$. Let $|\pi|=\max_{i=1,\ldots,m}|\mu_A(L_i)|$. We obtain a filter of $\mathcal B$-partitions
along which limits are taken.
By definition of the integral with respect with a measure we have:
\[
\int_Xf(x)\chi_{\left\{A(\w)\in dx\right\}}=\lim_{|\pi|\rightarrow 0}\sum_{i=1}^mf(x_i)\chi_{\left\{A(\w)\in L_i\right\}}.
\]
But 
\[
\begin{split}
\mathbb E\left(\sum_{i=1}^mf(x_i)\chi_{\left\{A(\w)\in L_i\right\}}\right)^2&=\sum_{i=1}^mf(x_i)^2\mu_A(L_i)\\
&\rightarrow\int_Xf(x)^2d\mu_A(x),\quad\text{as $|\pi|\rightarrow 0$},
\end{split}
\]
and the result follows.
\end{proof}

We now consider the case of a discrete random variable. We shall assume that $A\,:\,\Omega\,\longrightarrow\, \mathbb N_0$. So, 
$X=\mathbb N_0$, and the space $\mathbf L_2(X,\mathcal B_X,
\mu_A)$ is the Hilbert space of $\ell_2(\mu_A)$ real-valued sequences $(\xi_n)_{n\in\mathbb N_0}$ such that
\begin{equation}
\sum_{n=0}^\infty\xi_n^2\,\mathbb P\left(\left\{A=n\right\}\right)<\infty.
\end{equation}
We have
\begin{equation}
\left(M_{f\circ A}\psi\right)(\w)=\sum_{n=0}^\infty \xi_n\chi_{\left\{A=n\right\}}(\w)\psi(\w),\quad\forall \psi\in\mathbf L_2(\Omega,\mathcal F,\mathbb P),\,\,\forall f=(\xi_n)_{n\in\mathbb N_0}\in\ell_2(\mu_A),
\end{equation}
and 
\begin{equation}
\mathscr E_A\left(\left\{n\right\}\right)=M_{\chi_{\left\{A=n\right\}}},\quad n=0,1,\ldots.
\end{equation}

\begin{theorem}
The following formulas hold:
\begin{equation}
\mathbb E\left(\psi\,\big|\,\mathcal F_A\right)(\w)=
\begin{cases}\int_X\left(V_A^*\psi\right)(x)\chi_{\left\{A\in dx\right\}}(\w)\quad\hspace{3.8cm}(\text{continuous case})\\
\sum_{k=0}^\infty\frac{1}{\mathbb P(\left\{A=k\right\})}\left(\int_{\left\{A=k\right\}}\psi(\w)d\mathbb P(\w)\right)\chi_{\left\{A=k\right\}}(\w)\quad(\text{discrete case}).\end{cases}
\end{equation}
\end{theorem}
\subsection{Markov processes}
We follow the notation of the previous section, but our starting point is now a sequence of $X$-valued random variables $T_0,T_1,\ldots$
defined on the given probability space $(\Omega,\mathcal F,\mathbb P)$.

\begin{axioms}
\label{axioms123}
$(i)$ Let $\mathcal G_n\subset\mathcal F$ be the smallest sigma-algebra for which the variables $T_0,\ldots ,T_n$ are measurable. We have
that $V_{n}^*V_{n+1}$ does not depend on $n$, and (see \eqref{elisabeth} )
\begin{equation}
\label{markovprop}
\mathbb E\left(f\circ T_{n+1}\big|\mathcal G_n\right)=\mathbb E\left(f\circ T_{n+1}\big|\mathcal F_n\right)=R(f)\circ T_n,\quad n=0,1,\ldots
\end{equation}
$(ii)$ The measures $\mu_0$ and $\mu_1$ are equivalent.
\end{axioms}
We refer to \eqref{markovprop} as the Markov property in the present setting.

\begin{remark}
{\rm
If in the expression $R_{A,B}f=V_A^*V_Bf$, to $A=T_n$ and $B=T_{n+1}$ and if moreovoer $R_{T_{n+1},T_n}$ is independent of $n$ we get
\[
\mathbb E\left(f\circ T_{n+1}\,\big|\,\mathcal F_n\right)=\left(R(f)\right)\circ T_n
\]
as a special case of \eqref{elisabeth}. Iterating we get
\[
\mathbb E\left(f\circ T_{n+k}\,\big|\,\mathcal F_n\right)=\left(R^k(f)\right)\circ T_n.
\]
}
\end{remark}

\begin{lemma}
\label{wzero}
Condition $(ii)$ from Axioms \ref{axioms123} holds if and only if $\mu_0\left(\left\{x\,:\, W(x)=0\right\}\right)=0$.
\end{lemma}
\begin{proof}
Since $\mu_1<<\mu_0$ we can write
\begin{equation}
\label{1-1-1}
\mu_1(\Delta)=\int_\Delta W(x)d\mu_0(x),\quad \forall \Delta\in\mathcal B_X,
\end{equation}
where $W=\frac{d\mu_1}{d\mu_0}$. Let $\Delta_0=\left\{x\in X\,;\, W(x)=0\right\}$. By \eqref{1-1-1} we have $\mu_1(\Delta_0)=0$. Assume that
$\mu_0(\Delta_0)>0$. Then $\mu_0<<\mu_1$ will not hold.\smallskip

Conversely, if $\mu_0(\Delta_0)=0$, then $W^{-1}$ is well defined $\mu_0$ a.e., and $\mu_0<<\mu_1$   with $\frac{d\mu_0}{d\mu_1}=\frac{1}{W}$.
\end{proof}

\begin{definition}{\rm
Assume the previous axioms in force, and set $W=\frac{d\mu_1}{d\mu_0}$. The sequence $T_0,T_1\ldots$ is called a $W$-Markov process.}
\end{definition}

Given $\prod_0^\infty X$, we denote by $\pi_n$ the $n$-th coordinate function:
\[
\pi_n(x_0,x_1,\ldots)=x_n
\]
sent work, this product is always endowed with the cylinder sigma-algebra $\mathcal C$.

\begin{theorem}
Let $(\Omega,\mathcal F,\mathbb P,(T_n)_{n\in\mathbb N_0})$ satisfy axioms $(i)$ and $(ii)$ above. Then there is a probability measure $\mathbb P^\times$
on the Cartesian product $\prod_0^\infty X$, and an isomorphism $\widehat{T}$  between $(\Omega,\mathcal F,\mathbb P,(T_n)_{n\in\mathbb N_0})$ 
and $(\prod_{n=0}^\infty X,\mathcal C,\mathbb P^\times, (\pi_n)_{n\in\mathbb N_0})$, meaning that
\begin{equation}
\pi_n\circ\widehat{T}=T_n,\quad n=0,1,\ldots
\end{equation}
\label{michelle123}
\end{theorem}

\begin{proof} We define
\begin{equation}
\widehat{T}(\w)=(T_0(\w),T_1(\w),\ldots)
\end{equation}
and 
\begin{equation}
\mathbb P^\times (\Delta)=\mathbb P\left(\widehat{T}^{-1}(\Delta)\right),\quad \forall \Delta\in\mathcal C,
\end{equation}
in other words, $\mathbb P^\times$ is the distribution of $\widehat{T}$.
\end{proof}
\subsection{Discrete case}
We first compute the transfer operator (see \eqref{lycee_llg}) for a pair of  random variables $A$ and $B$ on $(\Omega,\mathcal F,\mathbb P)$  when $A$ is 
discrete. See Section \ref{conditional} for the notation. We denote by $\delta_n$ the Dirac function on $\mathbb N_0$, that is
\[
\delta_n(m)=\begin{cases}\, 1,\quad{\rm if}\quad m=n,\\
                         \,0,\quad{\rm if}\quad m\not=n.
\end{cases}
\]
\begin{proposition}
Let $(\Omega,\mathcal F,\mathbb P)$ and $A,B$ as above. Then the transfer operator $R_{B,A}\,:\, \ell_2(\mu_A)\,\longrightarrow\,
\mathbf L_2(X,\mathcal B_X,\mu_B)$ is  given by:
\begin{equation}
\left(R_{B,A}\left(\delta_n\right)\right)(x)=\mathbb E_{B=x}\left(\left\{A=n\right\}\,\big|\,\mathcal F_B\right).
\label{Adiscrete}
\end{equation}
\end{proposition}

\begin{proof} The result is immediate from Lemmas \ref{lemma22} and \ref{lemma24}. We get for $x\in X$ and $n\in\mathbb N_0$
\[
\begin{split}
\left(\left(V_B^*V_A\right)\delta_n\right)(x)&=\left(V_B^*\left(\chi_{\left\{A=n\right\}}\right)\right)(x)\\
&=\mathbb E_{B=x}\left(\left\{A=n\right\}\,|\,\mathcal F_B\right)
\end{split}
\]
where we have identified the indicator function $\chi_{\left\{A=n\right\}}$ on $\Omega$ with the subset 
\[
\left\{A=n\right\}=\left\{\w\in\Omega\,;\, A(\w)=n\right\}.
\]
\end{proof}

Note that, since  $R_{B,A}^*=R_{A,B}$, by Lemma \ref{lemma24} we get for $f\in\mathbf L_2(X,\mathcal B_X,\mu_B)$
\begin{equation}
\label{110616}
\left(R_{A,B}(f)\right)(n)=\int_Xf(x)\mathbb E_{B=x}\left(\left\{A=n\right\}\,\big|\, \mathcal F_B\right)d\mu_B(x).
\end{equation}

\begin{proposition}
\label{prop2}
On $\mathbb N_0\times\mathbb N_0$ we have the following positive definite kernel:
\begin{equation}
k_B(n,m)=\langle R_{A,B}\delta_n,R_{A,B}\delta_m\rangle_{\mathbf L_2(X,\mathcal B_X,\mu_B)}=\int_\Omega \chi_{\left\{A=n\right\}}(\w)\mathbb E_B\left(
\chi_{\left\{A=m\right\}}
\right)(\w)d\mathbb P(\w).
\end{equation}
\end{proposition}
\begin{proof}
This follows from \eqref{110616} and from the formula $R_{B,A}^*R_{B,A}=V_A^*\mathbb E_B V_A$
\end{proof}

The last result concerns the case where both $A$ and $B$ are discrete.

\begin{proposition}
Let $(\Omega,\mathcal F,\mathbb P)$ and $A,B$ as above, and assume that both $A$ and $B$ have discrete laws, with $\mu_A$ and $\mu_B$ both 
supported on the same countable discrete set, say $S$. Denote by $\delta^{(B)}_j$ the Dirac function viewed as a vector in $\ell_2(\mu_B)$. Then
\begin{equation}
\left(R_{A,B}\left(\delta_j^{(B)}\right)\right)(i)=\mathbb P\left(B=j\,\big|\, A=i\right).
\end{equation}
\end{proposition}

\begin{proof}
This is immediate from Lemma \ref{lemma26}.
\end{proof}
\begin{table}[h!]
\caption{Markov chains}
\begin{tabular}{|c|c|c|}
\toprule
& &\\
\text{Case}&\text{General state space $\mathcal M(X,\mathcal B_X)$}& \text{Discrete state space}\\
&& \\
\hline
& &\\
\text{Transition}&$\mathbb E\left(f\circ T_{n+1}\,\big|\,\bigvee_{j=0}^n \mathcal F_j\right)=$&
$\mathbb E\left(T_{n+1}=j\,\big|\, T_0=i_0,\ldots T_n=i_n\right)=$\\
&&\\
&$\hspace{1.8cm}=\mathbb E\left(f\circ T_{n+1}\,\big|\,\mathcal F_n\right)$&$\hspace{1.8cm}=\mathbb E\left(T_{n+1}=j\,\big|\,T_n=i_n\right)$\\
& &\\
\hline
&&\\
\text{Transfer}&$\left(Rf\right)(x)=\mathbb E\left(f\circ T_{n+1}\,\big|\, T_n=x\right)$& $p_{i,j}=\mathbb P
\left(T_{n+1}=j\,\big|\, T_n=i\right)$\\
\text{operator}& &\\
&$\forall f\in\mathcal M(X,\mathcal B_X)$&\\
\hline
\text{Harmonic}&&\\
\text{functions}& $(Rh)(x)=h(x)$ &$\sum_{j}p_{i,j}h_j=h_i$\\ 
& &\\
\bottomrule
\end{tabular}
\label{table1234}
\end{table}

\subsection{Martingales}
\begin{definition}{\rm
Let $T_0,T_1,\ldots$  be a Markov chain, with each $T_n$ taking values in the space $(X,\mathcal B_X)$. Let $M_0,M_1,\ldots$ be another 
$X$-valued random process. We say that $(M_n)_{n\in\mathbb N_0}$ is a martingale with respect to $(T_n)_{n\in\mathbb N_0}$ if the condition
\begin{equation}
\mathbb E\left(M_{n+k}\,\big|\,\bigvee_{j=0}^n\mathcal F_{T_j}\right)=M_n,\quad \forall n,k\in\mathbb N_0
\end{equation}
holds.}
\label{defmartingale}
\end{definition}

\begin{proposition}
Let $(X,\mathcal B_X)$ be a measure space, and let $T_0,T_1,\ldots$ be a $X$-valued Markov chain, with transition operator $R$ acting on
$\mathcal M(X,\mathcal B_X)$, and let $h$ positive on $X$ and such that $Rh=h$. Then, the process $M_n=h\circ T_n$, $n=0,1,\ldots$ is a
martingale (see Definition \ref{defmartingale} for the latter) with respect to $(T_n)_{n\in\mathbb N_0}$.
\end{proposition}

\begin{proof}
Since $(T_n)_{n\in\mathbb N_0}$ is a Markov chain we have for every $f\in\mathcal M(X,\mathcal B_X)$
\begin{equation}
\label{22222}
\mathbb E\left(f\circ T_{n+k}\,\big|\,\bigvee_{j=0}^n\mathcal F_j\right)=\mathbb E\left(f\circ T_{n+k}\,\big|\,\mathcal F_n\right)=
\left(R^k(f)\right)\circ T_n,
\end{equation}
(by Lemma \ref{elisabeth} with $B=T_{n+k}$ and $T_n$ instead of $A$; see also \eqref{martingale1} for the solenoid).\smallskip

Setting $f=h$ in \eqref{22222} we get
\[
\begin{split}
\mathbb E\left(M_{n+k}\,\big|\,\bigvee_{j=0}^n\mathcal F_j\right)&=(R^k(h))\circ T_n\\
&=h\circ T_n\\
&=M_n,
\end{split}
\]
which is the desired conclusion.
\end{proof}

\begin{remark}
{\rm The same argument will hold for a positive function $f$ such that $Rf=\lambda f$ for some $\lambda\not=0$. Then,
$M^{(\lambda)}_n=\lambda^{-n}f\circ T_n$ is a martingale with respect to  $(T_n)_{n\in\mathbb N_0}$.}
\end{remark}

\begin{definition} (stopping time) {\rm A stopping time $K$  for a random process $(T_n)_{n\in\mathbb N_0}$ on the probability space $(\Omega,
\mathcal F,\mathbb P)$ is a random variable $K\,:\,\Omega\,\longrightarrow\,\mathbb N_0$ such that 
\begin{equation}
K^{-1}\left(\left\{n\right\}\right)\in\bigvee_{j=0}^n\mathcal F_{T_j},\quad\forall,n\in\mathbb N_0.
\end{equation}
\label{stoppingtime}}
\end{definition}

Applying our previous analysis to the pair $A=(T_0,\ldots , T_n)$  (with associated space $X^{n+1}$) and $B=K$, we have the following stopping
time formula:
\begin{equation}
\left(R_{\left\{T_0,\ldots, T_n\right\},K}\right)(m)=\mathbb E_{\left\{K=m\right\}}\left(f(T_0,\ldots, T_n)\,\big|\,\mathcal F_K\right),\quad n,m
\in\mathbb N_0,\quad\forall f\in\mathbf L_2\left(\left(\prod_{n=0}^nX\right),\mu_n\right),
\end{equation}
where $\mu_n$ denotes the joint distribution of $\left\{T_0,\ldots,T_n\right\}$.
\begin{definition}
{\rm Let $(X,\mathcal B_X)$ and $(Y,\mathcal B_Y)$ be two measure spaces, and let $F\,:\, X\times Y\,\longrightarrow\, X$ be a measurable
function, where $X\times Y$ as a measure space is given the product sigma-algebra. Let $\left(T_n\right)_{n\in\mathbb N_0}$ be a
Markov chain with values in $X$, and let $(\psi_n)_{n\in\mathbb N_0}$ be a system of independent identically distributed (i.i.d.) random variables
with values in $Y$. If
\begin{equation}
\label{feed}
T_{n+1}=F(T_n,\psi_n),\quad n\in\mathbb N_0,
\end{equation}
then one says that $\left(T_n\right)_{n\in\mathbb N_0}$ is a homogeneous Markov chain (HMC). In details, the requirement is that
\[
T_{n+1}(\w)=F(T_n(\w),\psi_n(\w)),\quad \w\in\Omega,
\]
where $\Omega$ refers to the sample space in the probability space $(\Omega,\mathcal F,\mathbb P)$ which realizes the two processes; see
also Theorems \ref{loudmila} and \ref{thhmc} below.}
\label{227}
\end{definition}

Recursion \eqref{feed} is a feedback loop in the language of (non-linear)  system  theory. See Figure \ref{fig1} below, and see
e.g. \cite{MR525380} for information on feedback.
We plan to explore these connections in a future publication.\smallskip

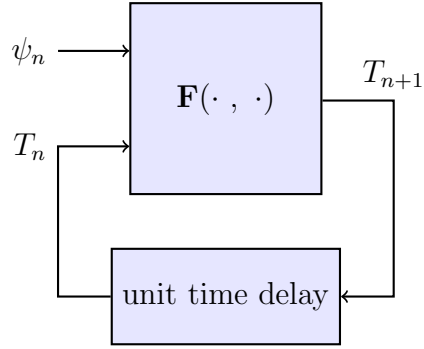
\begin{figure}
	
\hspace{4cm}	 \begin{tikzpicture}
		\matrix (diag_mat) [nodes={draw, fill=blue!10}, row sep=7mm, column sep=7mm]
		{ 
		\coordinate (taun) ; &
		\node[block1] (F) {$\mathbf{F}(\cdot~,~\cdot)$};&
		\coordinate (tnp1) {}; & 
		\coordinate (XB) {}; & 
		~;& 	\\	
		\coordinate (zr) ;&
		\node[block2] (G) {${\rm unit~time~delay}$};
		\coordinate (z0) ; &
		\coordinate (z1) ; &
		~ ;\\ 
	};	
	\draw[connector]  ($(taun)-(0,-0.25in)$) node[left] {$ֿ
\psi_n$}--($(F.west) - (0,-0.25in)$) ;
;
	\draw[connector]  ($(F.east) - (0,0.01in)$) -- ($(tnp1)-(0,0.01in)$) node[above] {$ֿT_{n+1}$}  -- (z1) 
-- 
(G)
;
	\draw[connector]  
(G) -- (zr) -- ($(taun)-(0,0.25in)$) node[left] {$ֿ
T_n$} -- ($(F.west) - (0,0.25in)$);	
	\end{tikzpicture}
\caption{Feedback loop: illustration of the class of i.i.d. feedback processes from Definition \ref{227}.}
\label{fig1}
\end{figure}

There are many applications of these Markov processes, including to control, see \cite{MR544839}, to feedback, 
see \cite{MR1974383}, and to Monte Carlo simulation, see e.g., \cite{MR1689633} and \cite[\S 5.5]{MR1600720}.
 
\subsection{Homogeneous Markov chains (HMC)}
As above $(X,\mathcal B_X)$ is a set  with a fixed sigma-algebra $\mathcal B_X$. We consider another measure-space 
$(Y,\mathcal D)$. Let $\psi_0,\psi_1,\ldots$ be a sequence of independent identically  distributed (i.i.d.) $Y$-valued random variables 
defined on $(\Omega,\mathcal F,\mathbb P)$, with probability distribution $\nu$. Such a sequence is called a white noise or a driving sequence; see
Remark \ref{wnoise}. One way to construct such a sequence is as follows. One takes 
\[
\Omega_Y=\prod_{n=0}^\infty Y\,\left(=Y^{\mathbb N_0}\right)
\]
endowed with the cylinder sigma-algebra $\mathcal C$ (see for instance \cite{Ka48} for the latter), and the infinite product measure 
$\nu_\infty=\nu\times\nu\cdots$, and set
\[
\psi_n(y_0,y_1,\ldots)=y_n
\]

Thus
\begin{equation}
\nu(D)=\nu_\infty(\psi_n^{-1}(D)),\quad D\in\mathcal D,\,\,\,{\rm and}\,\,\, n=0,1,\ldots
\end{equation}
In particular,
\begin{equation}
\int_\Omega F(\cdot, \psi_0(\w))d\nu_\infty(\w)=\int_YF(\cdot, y)d\nu(y)
\end{equation}
We now consider a measurable map $F$ from $X\times Y$ into $X$, where $Y$ is another measure-space. We define
\begin{equation}
\label{roxanna123}
(R_Ff)(x)=\int_Yf(F(x,y))d\nu(y),\qquad f\in\mathcal M(X,\mathcal B).
\end{equation}

Let $\Omega_Y=\prod_{n=0}^\infty Y$, and for $\w=(y_0,y_1,\ldots)\in\Omega_Y$, we define (with $F_y=F(\cdot, y)$)
\begin{equation}
\label{wnfn}
\begin{split}
\w|n&=(y_0,\ldots, y_n)\\
F_{\w|_n}&=F_{y_n}F_{y_{n-1}}\cdots F_{y_1}F_{y_0}.
\end{split}
\end{equation}
We assume that
\begin{equation}
\label{11w}
\cap_{n=1}^\infty F_{\w|_n}(X)=\left\{x_\w\right\}
\end{equation}
is a singleton. We then set
\begin{equation}
\label{Vw}
V(\w)=x_\w \quad(\text{see \eqref{11w}}).
\end{equation}

\begin{lemma}
Let $F\,:\, X\times Y\,\longrightarrow\, X$. The corresponding transfer operator in \eqref{roxanna123} is of the 
form
\begin{equation}
(R_Ff)(x)=\int_Xf(t)\mu(dt,|x)
\end{equation}
where $\mu(\cdot|x)\stackrel{\rm def.}{=}(d\nu)\circ F_x^{-1}$, and where $F_x(\cdot)=F(x,\cdot)\,:\, Y\,\longrightarrow\, X$.
\end{lemma}
\begin{proof}
We have
\[
\begin{split}
\int_Yf(F(x,y))d\nu(y)&=\int_Y(f\circ F_x)(y)d\nu(y)\\
&=\int_Xf(t)\left(d\nu\circ F_x^{-1}\right)(t).
\end{split}
\]
Hence $\mu(\cdot|x)=d\nu\circ F_x^{-1}$ as claimed.
\end{proof}
\begin{theorem}
Let $\nu$ be a probability measure on $(Y,\mathcal D)$, and let $\times_{n=0}^\infty \nu$ be the corresponding infinite product measure on $\Omega$. Assume 
that \eqref{11w} is in force, and let $V$ be defined by \eqref{Vw}. The formula
\begin{equation}
\mu(B)=\left(\times_{n=0}^\infty \nu\right)(V^{-1}(B)),\quad B\in\mathcal B_X,
\end{equation}
then defines a measure on $(X,\mathcal B)$ which satisfies
\begin{equation}
\label{equationRF}
\mu R_F=\mu,
\end{equation}
that is
\begin{equation}
\label{conclusion13} 
\iint_{X\times Y}f(F(x,y))d\nu(y)d\mu(x)=\int_Xf(x)d\mu(x),\quad \forall f\in\mathcal M(X,\mathcal B_X)
\end{equation}
holds.
\label{loudmila}
\end{theorem}

\begin{proof}
We define on $\Omega$ 
\[
\ell(y)(y_0,y_1,\ldots)=(y,y_0,y_1,\ldots),\quad{\rm with}\,\, y\in Y\,\,{\rm and}\,\, \w=(y_0,y_1,\ldots)\in\Omega.
\]
Then it is clear from \eqref{11w} and \eqref{Vw} that 
\begin{equation}
\label{commu}
F_yV=V\ell(y),
\end{equation}
since
\[
F_yF_{\w|n}=F_{\ell(\w)|n+1}.
\]
Note that \eqref{commu} means that the following commutative diagram is in force:
\[
\begin{array}{ccc}
\Omega_Y&\xrightarrow{\hspace*{0.5cm}{V}\hspace*{.5cm}}&X \\
\hspace*{-0.5cm}\downarrow{\hspace*{-1.2cm}\ell(y)}
& &\hspace*{0.5cm}\downarrow{F_y}\\
\Omega_Y&\xrightarrow{\hspace*{0.5cm}{V}\hspace*{.5cm}}&X
\end{array}\]
We now prove \eqref{conclusion13}. Let $f\in\mathcal M(X,\mathcal B)$. Then $R_F$ in \eqref{equationRF} can be rewritten as
\[
R_Ff=\int_Y (f\circ F_y)d\nu(y).
\]
With $\mathbb Q=\times_{n=0}^\infty \nu$ and $\mu=\mathbb Q\circ V^{-1}$ we get
\[
\begin{split}
\int_X (R_Ff)(x)d\mu(x)&=\int_X(R_Ff)(x)(d\mathbb Q\circ V^{-1})(x)\\
&=\int_{\Omega_Y}((R_Ff)\circ V)(\w)d\mathbb Q(\w)\\
&=\int_{\Omega_Y}\int_Y((f\circ F_y)\circ V)(\w)d\mathbb Q(\w)d\nu(y)\\
&(\text{    and using \eqref{commu}})\\
&=\int_{\Omega_Y}\int_Y(f\circ V\circ \ell(y))d\mathbb Q(\w)d\nu(y)\\
&=\int_{\Omega_Y}\int_Y(f\circ V)(\w)d\mathbb Q\circ \ell(y)^{-1}d\nu(y)\\
&(\text{    and since $\mathbb Q$ is an infinite product measure})\\
&=\int_{\Omega_Y}\int_Y(f\circ V)(\w)d\mathbb Q(\w)d\nu(y)\\
&=\int_Xf(x)\int_Y(d\mathbb Q\circ V^{-1})(x)d\nu(y)\\
&=\int_Xf(x)(d\mathbb Q\circ V^{-1})(x)\\
&(\text{    since $d\nu$ is a probability measure})\\
&=\int_X f(x)d\mu(x)\quad (\text{with $\mu=\mathbb Q\circ V^{-1}$})
\end{split}
\]
\end{proof}
\begin{theorem}
Let $Y$ and $F$ be as above. Let $\lambda$ be a probability measure on $(X,\mathcal B)$, and let $\psi_0,\psi_1,\ldots$ be a sequence of 
i.i.d. $Y$-valued random variables with probability distribution $\nu$. Then there exists a probability measure $\mathbb P$
on $(\Omega_Y,\mathcal C)$ and a sequence of 
$X$-valued random variables $T_0,T_1,\ldots$ on $\Omega_Y$ such that:\smallskip

\hspace{1cm}$(1)$ $\lambda$ is the distribution of $T_0$, that is
\begin{equation}
\int_\Omega F(T_0(\w),\cdot)d\mathbb P(\w)=\int_XF(x,\cdot)d\lambda(x).
\end{equation}
\hspace{1cm}$(2)$ We have
\begin{equation}
T_{n+1}=F(T_n,\psi_n),\quad n=0,1,\ldots
\label{agathe}
\end{equation}
\hspace{1cm}$(3)$ It holds that
\begin{equation}
\mathbb E\left(f\circ T_{n+1}\,\big|\,\mathcal F_n\right)=\mathbb E\left(f\circ T_{n+1}\big|\mathcal G_n\right)=
(R_F(f))\circ T_n
\end{equation}
where $\mathcal F_n=T_n^{-1}(\mathcal B)$, where $\mathcal G_n$ is the smallest sigma-algebra for which the variables $T_0,\ldots, T_n$ 
are measurable, and where $R_F$ is given by \eqref{roxanna123}.\smallskip

\hspace{1cm}$(4)$ We have
\begin{equation}
\label{proba123}
\int_{\Omega_Y}(f_0\circ T_0)(f_1\circ T_1)\cdots (f_n\circ T_n)d\mathbb P=\int_Xf_0(x)R_F(f_1R_F(f_2\cdots R_F(f_nh)\cdots))d\lambda(x)
\end{equation}
with $f_0,f_1,\ldots, f_n\in \mathcal M(X,\mathcal B)$.
\label{thhmc}
\end{theorem}

\begin{remark}
\label{wnoise}
{\rm 
Equation \eqref{agathe} is called an homogeneous Markov chain driven by white noise. The sequence $\psi_0,\psi_1,\ldots$
is called the driving sequence. See \cite[p. 56]{MR1689633}. For general background on  time-homogeneous state equation and 
homogeneous Markov chains, \cite{MR544839,MR1974383}. See also Theorem \ref{thmmm} below.}
\end{remark}

\begin{remark}
\label{utopia}
{\rm When $R$ is not normalized one defines
\[
R^\prime (f)=\frac{R(fh)}{h}.
\] 
Then, $R^\prime 1=1$. The above construction applied to the pair $(R^\prime,hd\lambda)$ will lead to the same 
probability measure $\mathbb P$. This is because
\[
\int_Xf_0(x)R(f_1R(f_2\cdots R(f_nh)\cdots))d\lambda(x)=
\int_Xf_0(x)R^\prime(f_1R^\prime(f_2\cdots R^\prime(f_n)\cdots))h(x)d\lambda(x).
\]
}
\end{remark}
\begin{proof}[Proof of Theorem \ref{thhmc}] The proof is  divided into three steps, which we outline.\\

STEP 1: {\sl Let $\w=(y_0,y_1,\ldots)\in\Omega_Y$ and define $\psi_n(\w)=y_n$, and $T_0,T_1,\ldots$ via
\begin{equation}
T_n(\w)=F(\cdots (F(F(F(T_0(\w),y_0),y_1),\cdots ),y_{n-1}),\ldots).
\end{equation}
}

STEP 2: {\sl Formula \eqref{proba123} defines a unique probability measure $\mathbb P$ on $\Omega_Y$ endowed 
with its cylinder sigma-algebra.}\\

The existence of $\mathbb P$ is an application of Kolmogorov's consistency principle; see for instance \cite{MR0203748}.\\

STEP 3: {\sl The above probability, and the random functions $T_0,T_1,\ldots$ have the desired properties.}\\

See also Lemma \ref{lemma22} and \ref{lemma24}.

\end{proof}

\begin{corol}
The probability distribution of $T_n$  is $\mu_n=\lambda R^n$, $n=1,2,\ldots$
\end{corol}

\begin{proof}
This follows from Theorem \ref{thhmc}. See also Theorem \ref{michelle123}.
\end{proof}

\begin{corol}
\begin{equation}
\mu_n(B)=(\lambda\times \underbrace{\nu\times\cdots \times \nu}_{\text{$n$ times}})(F_n^{-1}(B)),\quad B\in\mathcal B_X.
\end{equation}
where 
\begin{equation}
\begin{split}
F_n(x,y_1,y_2,\ldots, y_n)&=(F_{y_n}\cdots F_{y_1})(x)\\
&=F(\cdots F(F(F(x,y_1),y_{2})\cdots ,y_{n-1}),y_n).
\end{split}
\end{equation}

     \end{corol}

We set
\[
\pi_1(x,y)=x.
\]

\begin{corol}
Let $B\in\mathcal B$ and $x\in X$. Then
\begin{equation}
\mathbb E\left(T_{n+1}\in B\,\big|\, T_n=x\right)=\nu(\pi_1^{-1}(x)\cap F^{-1}(B))
\end{equation}
\end{corol}

\begin{proof}
The result follows from Lemmas \ref{lemma22} and \ref{lemma24}.
\end{proof}

Background references on multiresolutions include \cite{MR2559716,MR2888226,BrJo02a,MR1162107,MR1162107}.

\subsection{Multiresolutions and Cuntz-Krieger relations}
As a corollary of the previous analysis we now consider the case where possibly more than two random variables are given. 

\begin{corol}
\label{corol224}
Given $N$ random variables $A_1,\ldots, A_N$ with values in $X$, ($N=\infty$ is allowed),
the following hold:
\begin{eqnarray}
\label{c1}
V_{A_u}^*V_{A_v}&=&R_{u,v}\quad (\text{definition of the transfer operator from $A_u$ to $A_v$})\\
\label{c2}
V_{A_u}^*V_{A_u}&=&I_{\mathbf L_2(\mu_u)},\quad u=1,2,\ldots, N\\
V_{A_u}V_{A_u}^*&=&\mathbb E\left(\cdot\,\big|\,\mathcal F_{A_u}\right),\quad u=1,2,\ldots, N\\
\label{c4}
\sum_{u=1}^N V_{A_u}V_{A_u}^*&=&\mathbb E\left(\cdot\,\big|\,\cup_{u=1}^N\mathcal F_u\right),\\
\sum_{u=1}^N V_{A_u}V_{A_u}^*&=&\mathbb E\left(\cdot\,\big|\,\mathcal F\right)\,=I\,\,\,\, if\,\,\,\,\,\mathcal F =\cup_{n=1}^N \mathcal F_{A_n}.
\label{c5}
\end{eqnarray}
If $N=\infty$, the latter sums \eqref{c4}-\eqref{c5} converge in the strong operator topology.
\end{corol}

\begin{proof}
This follows from Corollary \ref{corol25} and Table \ref{rubicon}.
\end{proof}

\begin{remark}{\rm
Relations \eqref{c1}-\eqref{c5} can be seen as a generalization of the Cuntz-Krieger relations (see \cite{Cun77,MR592141} for the latter), 
and they lead to a multiresolution decomposition of the probability space $\mathbf L_2(\Omega,\mathcal F,\mathbb P)$; see Section \ref{unitcircle}.
A practical interpretation of formulae \eqref{c4} and \eqref{c5}, is the 
assertion that certain random variables may be reconstructed by samples. In this case, the sampling is performed with the use of random variables as specified in the premise in Corollary \ref{corol224}. For a practical use of related 
sampling formulas in learning theory, see e.g., \cite{MR2488871}.}
\end{remark}

\subsection{The Schur algorithm and homogeneous Markov chains (HMC)}
The Schur algorithm provides an application of the above analysis. We first recall the following (see 
\cite{MR1638044,MR94b:47022,C-schur,MR2004e:46003,MR1511896,schur}). 
Let $s$ be a function analytic and strictly contractive in the open unit disk
$\mathbb D$ (we will call such functions Schur functions, and denote their set by $\mathcal S$). Then, the functions $s_1,\ldots$ defined recursively by
$s_0(z)=s(z)$ and
\begin{equation}
s_{n+1}(z)=\dfrac{s_n(z)-s_n(0)}{z(1-\overline{s_n(0)}s_n(z))},\quad n=0,1,\ldots
\end{equation}
belong to $\mathcal S$ as long as $|s_n(0)|<1$. The recursion stops at rank $n$ if $|s_n(0)|=1$. As already proved by Schur, this  will happen if and
only if $s$ is a finite Blaschke product. The numbers $\rho_n=s_n(0)$, $n=0,1,\ldots$ are called the Schur parameters of $s$, and determine uniquely
the function $s$ in terms of a partial fraction exansion
\begin{equation}
\label{shelly}
s(z)=\rho_0+\dfrac{z(1-|\rho_0|^2)}{\overline{\rho_0}z-\dfrac{1}{\rho_1+\dfrac{z(1-|\rho_1|^2)}{\overline{\rho_1}-\cdots}}}
\end{equation}
See \cite[p.   285]{wall}. When $s$ is a finite Blaschke product, the sequence is finite, and its last element is of modulus $1$.\\

Set $X=\mathcal S\setminus\left\{\mbox{\text unitary constants and finite Blaschke products}\right\}$ and $Y=\mathbb D$. We define
\begin{equation}
F(s,\rho)(z)=\frac{s(z)-\rho}{z(1-s(z)\overline{\rho})}
\end{equation}
which maps $X\times Y$ into $X$. We will also  use the notations $F_\rho(s)$ and $(F_\rho(s))(z)$. 
We set $\Omega=\prod_{n=0}^\infty \mathbb D$ and,
\[
\w=(\rho_0,\rho_1,\ldots)\quad{\rm and}\quad \w|_n=(\rho_0,\ldots, \rho_n).
\]
Furthermore, we define (see \eqref{wnfn})
\[
(F_{\w|_n})(s)=\left( F_{\rho_n}F_{\rho_{n-1}}\cdots F_{\rho_1}F_{\rho_0}\right)(s),\quad \text{see \eqref{shelly}.}
\]
We denote by $V$ the map
\begin{equation}
\label{Vdef}
V(\w)=s_\w
\end{equation}
where $s_\w\in\mathcal S$ is uniquely defined element from $\w$ via \eqref{shelly}.

\begin{lemma}
\label{Wis}
For every $\w\in\Omega$ we have
\begin{equation}
\cap_{n=0}^\infty F_{\w|_n}(\mathcal S)=\left\{s_\w\right\}
\end{equation}
where $s_w=V(\w)$, see \eqref{11w}.
\begin{proof}
This follows from the fact that a given Schur function is uniquely determined by the sequence of Schur coefficients when the latter is 
infinite. See \cite{schur}.
\end{proof}
\end{lemma}

As a consequence of Theorem \ref{loudmila} we have the following result. In the proof the transfer operator now  takes the form
as in \eqref{roxanna123}, that is
\begin{equation}
\label{roxanna2}
((R_Ff)(s))(z)=\int_{\mathbb D}f((F(s,\rho))(z))d\nu(\rho).
\end{equation}

\begin{theorem}
\label{danielle}
Let $\nu$ be a probability measure on $\mathbb D$ endowed with its Borel sigma-algebra, and let $\mathbb Q=\mathbb Q_\nu=\nu\times\nu\times\cdots$ be 
the corresponding infinite product measure on $\Omega=\prod_{n=0}^\infty \mathbb D$ endowed with the cylinder sigma-algebra.
Then $\mu=\mathbb Q_\nu\circ V^{-1}$, where $V$ is defined by \eqref{Vdef}, is a positive measure on $\mathcal S$ (or, more precisely, on the
set $\mathcal S$ from which the unitary constants and finite Blaschke products have been removed) satisfying
\begin{equation}
\mu R_F=\mu.
\end{equation}
\end{theorem}

\subsection{A summary of formulas}
We now summarize some of the formulas obtained in this section, pertaining to  two  given $X$-valued random variables $A$ and $B$
\begin{table}[h!]
\caption{Summary}
\begin{tabular}{|c|c|c|}
\toprule
& &\\
\text{Hilbert spaces}& $\mathbf L_2(\mu_A)\longrightarrow \mathbf L_2(\Omega,\mathbb P)$&
$\mathbf L_2(\Omega,\mathbb P)\longrightarrow    \mathbf L_2(\mu_A)$\\
& &\\
\hline
& &\\
\text{Operators}& $V_Af=f\circ A$&$(V_A^*\psi)(x)=\mathbb E_{A=x}\left(\psi\,\big|\,\mathcal F_A\right)$\\
& &\\
\hline
& &\\
\text{Hilbert spaces}& $\mathbf L_2(\Omega,\mathbb P)\longrightarrow \mathbf L_2(\Omega,\mathbb P)$& 
$\mathbf L_2(\mu_B)\longrightarrow \mathbf L_2(\mu_A)$\\
& &\\
\hline
& &\\
\text{Operators}& $V_AV_B^*\psi=\mathbb E\left(\psi\,\big|\,\mathcal F_B\cap\left\{A=B\right\}\right)$&
$(V_A^*V_Bf)(x)=\mathbb E_{A=x}\left(f\circ B\,\big|\,\mathcal F_A\right)$\\
& &\\
& (\text{need $V_B^*\psi\in\mathbf L_2(\mu_A)\cap\mathbf L_2(\mu_B)$})&\text{the tranfer operator $R_{A,B}$}\\
& &\\
\hline
& &\\
\text{Special case $A=B$}& $V_AV_A^*=\mathbb E\left(\cdot\,\big|\,\mathcal F_A\right)$& $V_A^*V_A=I_{\mathbf L_2(\mu_A)}$\\
& &\\
\hline
& &\\
\text{Product of the}& &\\
\text{conditional expectations}& $\mathbb E_{\mathcal F_A}\mathbb E_{\mathcal F_B}=V_AR_{A,B}V_B^*$& 
$\mathbf L_2(\Omega,\mathbb P)\longrightarrow\mathbf L_2(\mu_B)\longrightarrow$\\
$\mathbb E_{\mathcal F_A}$ and $\mathbb E_{\mathcal F_B}$& &\\
& & $\longrightarrow \mathbf L_2(\mu_A)\longrightarrow\mathbf L_2(\Omega,\mathbb P)$\\
\bottomrule
\end{tabular}
\label{rubicon}
\end{table}

The proofs of these various formulas are given in the section. See in particular Lemmas \ref{lemma22}, \ref{lemma24} and Corollary
\ref{corol25}.

\begin{remark}{\rm
In the case when $B$ is discrete, say $B\,:\,\Omega\,\longrightarrow\, \mathbb N_0$, the formula for $R_{A,B}$ simplifies as follows:
\[
\left(R_{A,B}f\right)(x)=\sum_{n=0}^\infty\frac{f(n)}{\mathbb P(\left\{B=n\right\})}\mathbb E_{\left\{A=x\right\}}\left(\chi_{\left\{B=n\right\}}\,\big|
\,\mathcal F_A\right),\quad \forall x\in X
\]
for functions $f\,:\, \mathbb N_0\,\rightarrow\,\mathbb R$ such that
\[
\sum_{n=0}^\infty|f(n)|^2\mathbb P(\left\{B=n\right\})<\infty,\quad i.e., f\in\ell_2(\mu_B).
\]
}
\end{remark}
\section{Solenoid probability spaces}
\setcounter{equation}{0}

{\bf Why the solenoids?} A number of reasons. Given an endomorphism  $\sigma$ in a measure space, the associated solenoid  
${\rm Sol}_\sigma$ is then a useful tool for the study of scales of multiresolutions (see Definitions \ref{sol} and \ref{newdef123}). The latter includes those resolutions arising naturally 
from discrete wavelet algorithms, as well as from the study of non-reversible dynamics in ergodic theory in and physics. In fact it is not so 
much ${\rm Sol}_\sigma$ itself that is central in this program, but rather probability spaces $({\rm Sol}_\sigma, \mathcal F, \mathbb P)$ 
where the solenoid is the sample space. It is the pair $(\mathcal F, \mathbb P)$  which carries the information about the relevant scales 
of multiresolutions for the problem at hand, and the nature and the details of $(\mathcal F, \mathbb P)$ change from one algorithm to the 
next; much like traditional wavelet analysis depend on scaling functions, father function, mother functions etc in $\mathbf L_2(\mathbb R^d)$.
But the latter is too restrictive a framework; see e.g. Section \ref{sec4} and \cite{MR1760275,BrJo02a,MR1162107}. 
See also Tables \ref{table1} and \ref{rubicon2}.

\smallskip

By ``discrete wavelet algorithms'' we mean recursive algorithms with selfsimilarity given by a 
scaling matrix. In one dimension, this may be just the $N$-adic scaling, but in general we allow for “discrete time” to be modelled by 
higher rank lattices, by more general discrete abelian groups, or even by infinite discrete sets with some given structure. For a given time-series, even in this general form, we may always introduce an associated generating function. 
This will be a function in ``dual frequency variables'' in one or more complex variables, and called
the frequency response function (see e.g., \cite{BrJo02a}). In many classical 
wavelet settings the given 
discrete wavelet algorithms may be realized in $\mathbf L_2(\mathbb R^d)$ for some $d$, but such a realization places very strong 
restrictions and limitations on the given multi-band filters making up the discrete wavelet algorithm at hand. We show that with the Hilbert 
space $\mathbf L_2({\rm Sol}_\sigma, \mathcal F, \mathbb P)$, we can get around this difficulty, and still retain the useful features of 
multi-scale resolutions and selfsimilarity which makes the wavelet realizations so useful.\smallskip

Motivated by multiresolutions in statistical computations, in many applications, and in particular in generalized wavelet algorithms, we study here a setting of dynamics of endomorphisms of measure spaces, denoting a given endomorphism by  $\sigma$, say acting in $X$
(see \cite{MR1793194,MR3441734,MR2599889,xMR2599889,MR1837681}). If the associated transfer operator $R$ is further given to be $\sigma$-homogeneous (see Definition \ref{julie} below), we show that the associated 
$R$-Markov processes will be of a special kind: when realized in the natural probability space of an associated solenoid ${\rm Sol}_\sigma$ computed from $\sigma$, we then arrive at natural multi-scale resolutions inside the Hilbert space  $\mathbf L_2({\rm Sol}_\sigma, 
\mathcal F,\mathbb P)$, with the scale of resolutions in question defined from the given endomorphism $\sigma$. In the case when 
$\sigma$ is the scale endomorphism of a wavelet construction, we show that the multi-scale resolution at hand will agree with that of the associated solenoid analysis.
And when a wavelet is realizable in Euclidean space, for example on the real line $\mathbb R$, then we show that then 
$\mathbb R$ is naturally embedded as a “curve” in the solenoid. Moreover, we identify the analogous multivariable setting with 
endomorphism and solenoid. Background references on analysis on solenoids and related multiresolutions include 
\cite{MR2888226,BrJo02a,MR3275999,MR2391805,MR1837681,MR3204025}.\smallskip

In our discussion of solenoids and multiresolutions, we have here restricted the discussion to the commutative case, as our motivation is 
from stochastic processes. But in the recent literature, there is also an exciting, and somewhat parallel non-commutative theory of 
solenoids and their multiresolutions. It too is motivated (at least in part) by developments in the analysis of wavelet-multiresolutions, and the corresponding scaling operators. However, the relevant questions in the non-commutative theory are quite different from those 
addressed here. The relevant questions are simply different in the non-commutative theory. The differences between the two in fact reflect 
the dichotomy for two different notions of probability theory, the difference between (classical) commutative, versus non-commutative 
probability theory. Among the recent papers on the non-commutative theory, we mention 
\cite{MR2559716,MR2609543,MR2888226,MR3204025,MR3441736}, and the literature cited there.
 
\subsection{Definitions}
Consider a locally compact Hausdorff space $X$, with associated Borel sigma-algebra $\mathcal B$, let $\sigma$ be a measurable endomorphism of $X$, which is
onto. We denote by $\mathcal M(X,\mathcal B)$ the space of all measurable functions from $X$ into $\mathbb R$.

\begin{definition}{\rm 
A map $R$ from $\mathcal M(X,\mathcal B)$ into itself is called a $\sigma$-{\rm transfer operator} (or a {\rm Ruelle operator}) if
\begin{equation}
\label{trans123}
Rf\ge 0,\quad \forall f\in \mathcal M(X,\mathcal B)\,\,{\rm satisfying}\,\, f(x)\ge 0,\,\,\forall x\in X,
\end{equation}
and if the {\rm pull-out} property
\begin{equation}
\label{roxanna}
R\left((f\circ\sigma )g\right)=fR(g),\quad \forall f,g\in\,\mathcal M(X,\mathcal B)
\end{equation}
holds.}
\label{julie}
\end{definition}

As a first example we have:
\begin{lemma}
Let $(X,\mathcal B)$ be a measure-space and let $A$ and $B$ be two $X$-valued random variables  on the probability space $(\Omega,
\mathcal F,\mathbb P)$, with transfer operator $R_{A,B}$ given by \eqref{lycee_llg}. Let $\sigma$ be an endomorphism of $X$ which is onto
and such that 
\begin{equation}
\label{notsigma}
\sigma\circ B=A.
\end{equation} 
Then $R_{A,B}$ satisfies the pull-out property \eqref{roxanna}. Moreover
\begin{equation}
R_{A,B}^*f=f\circ \sigma.
\label{adjointRAB}
\end{equation}
\end{lemma}

\begin{proof}
Indeed
\[
\begin{split}
\left(R_{A,B}((f\circ \sigma)g)\right)\circ A&=\mathbb E\left[((f\circ\sigma)g)\circ B\,\big|\, \mathcal F_A\right]\\
&=\mathbb E\left[(f\circ A)(g\circ B)\,\big|\, \mathcal F_A\right]\quad (\text{since $\sigma\circ B=A$})\\
&=(f\circ A)\mathbb E\left[g\circ B\,\big|\, \mathcal F_A\right]\\
&=\left(fR_{A,B}(g)\right)\circ A.
\end{split}
\]
We now prove \eqref{adjointRAB}. We have
\[
\begin{split}
\langle g,f\circ\sigma\rangle_{\mu_B}&=\int_X g(x)((f\circ\sigma)(x))d\mu_B(x)\\
&=\int_X R(g(f\circ\sigma))(x)d\mu_A(x)\\
&=\int_Xf(x)(R(g)(x))d\mu_A(x)\\
&=\langle R(g),f\rangle_{\mu_A}.
\end{split}
\]
\end{proof}

We note that $\mathcal F_A\subset\mathcal F_B$ when \eqref{notsigma} is in force. We now present an example of pairs of random variables for 
which neither $\mathcal F_A\subset\mathcal F_B$ nor $\mathcal F_B\subset\mathcal F_A$ hold. In particular they cannot be connected by an 
endomorphism of $X$.

\begin{example}
Consider the space $\left\{-1,1\right\}$ with probability distribution 
\[
p(\left\{1\right\})=p(\left\{-1\right\})=\frac{1}{2}.
\]
We take $\Omega=\prod_{n=1}^\infty \left\{-1,+1\right\}$, and $\mathbb P$ the corresponding infinite product measure on the
cylinder sigma-algebra. Let $a\in (0,1)$ and define
\begin{equation}
E_a(\w)=\sum_{k=1}^\infty \w_ka^k,
\end{equation}
where $\w=(\w_1,\w_2,\ldots)\in\Omega$ and thus $\w_k\in\left\{-1,1\right\}$.
The random variable $E_a$ takes values in $\mathbb R$, and its distribution, defined by
\[
\alpha_a(x_1,x_2)=\mathbb P(\w\in\Omega,\,x_1<E_a(\w)<x_2)
\]
has Fourier transform
\begin{equation}
\widehat{\alpha_a}(t)=\prod_{k=1}^\infty \cos (a^kt).
\end{equation}
It is known that (see \cite{MR2945155,MR3275999,MR1356783}:\smallskip

$(1)$ When $a<1/2$ the distributions $\alpha_a$ are singular, and mutually singular.\\
$(2)$ When $a=1/2$ we obtain the Lebesgue measure.\\
$(3)$ When $a\in(1/2,1)$ the corresponding $\alpha_a$ are absolutely continuous with respect to Lebesgue measure, for almost all values
of $a$. This is called the Erd\"os conjecture (see \cite{MR0000311,MR0000858}), and was proved in \cite{MR1356783}. The only known value of $a>\frac{1}{2}$ for which 
$\alpha_a$ is known not to be absolutely continuous with respect to Lebesgue measure is the reciprocal of the 
golden ratio $a=\frac{\sqrt{5}-1}{2}$. See also \cite[p. 48]{MR1669737} for further references and information.\smallskip

Taking $a_1$ and $a_2$ such that the corresponding distributions are mutually singular leads to random variables $E_{a_1}$ and 
$E_{a_2}$ which cannot be related by an endomorphism of $X$.
\end{example}

\begin{definition}
\label{defsol123}
{\rm The solenoid ${\rm Sol}_\sigma(X)$ associated with $\sigma$ is the subset of sequences $(x_k)_{k\in\mathbb N_0}$ in $X^{\mathbb N_0}$ such that
\begin{equation}
\label{sol}
\sigma(x_{k+1})=x_k,\quad k=0,1,\ldots
\end{equation}}
\end{definition}
\begin{remark}
{\rm
We think of a ``point'' in ${\rm Sol}_\sigma(X)$ as a path-governed by $\sigma$, and hence ${\rm Sol}_\sigma(X)$ as a path-space.}
\end{remark}
We set
\begin{equation}
\pi_k(x_0,x_1,\ldots)=x_k,\quad k\in\mathbb N_0\,\,\, {\rm and}\,\,\, (x_k)_{k\in\mathbb N_0}\,\in\,X^{\mathbb N_0},
\end{equation}
and \eqref{sol} can be rewritten as
\begin{equation}
\label{sol1}
\sigma\circ\pi_{n+1}=\pi_n,\quad n=1,2,\ldots
\end{equation}
The endomorphism $\sigma$ is (in general) neither one-to-one nor onto. But:

\begin{proposition}
The induced map $\widehat{\sigma}$ defined by
\begin{equation}
\label{sigmahat}
\widehat{\sigma}(x_0,x_1,\ldots)=(\sigma(x_0),x_0,x_1\ldots)
\end{equation}
is one-to-one from ${\rm Sol}_\sigma(X)$ onto itself, with inverse
\begin{equation}
\label{tau123}
\widehat{\sigma}^{-1}(x_0,x_1,\ldots)=(x_1,x_2,\ldots)
\end{equation}
\end{proposition}

\begin{proof} One-to-oneness is clear. Let $\tau$ denote the map in \eqref{tau123}. Then, 
\[
\widehat{\sigma}\circ\tau(x_0,x_1,\ldots)=\widehat{\sigma}(x_1,\ldots)=(\sigma(x_1),x_1,\ldots)=(x_0,x_1,\ldots)
\]
since $\sigma(x_1)=x_0$ and
\[
\tau\circ\widehat{\sigma}(x_0,x_1,\ldots)=\tau(\sigma(x_0),x_0,x_1,\ldots)=(x_0,x_1,\ldots).
\]
\end{proof}

We note that
\begin{equation}
\label{corinne}
\pi_0\circ\widehat{\sigma}=\sigma\circ\pi_0\quad{\rm and}\quad \pi_{n+1}\circ\widehat{\sigma}=\pi_n,\quad n=0,1,\ldots
\end{equation}

Recall that the notation $\mathcal F_A$ was introduced in \eqref{rachel}. We set
\[
\mathcal F_{\pi_n}=\mathcal F_n,
\] 

\begin{definition}
\label{fn}
{\rm
$\mathcal F_n=\pi_n^{-1}(\mathcal B)$ is the sigma-algebra generated by the random variables $f\circ\pi_n$, where $f$ runs through 
the measurable functions on $(X,\mathcal B)$.}
\end{definition}

As an immediate consequence of \eqref{corinne} we have:

\begin{lemma}
In the notation of Definition \ref{fn}, we have
\[
\mathcal F_0\subset\mathcal F_1\subset\cdots\subset\mathcal F_n\subset\mathcal F_{n+1}\subset\cdots
\]
and $\cup_{n=0}^\infty \mathcal F_n=\mathcal F$.
\end{lemma}

\begin{proof} Let $f\in\mathcal M(X,\mathcal B_X)$ and $n\in\mathbb N$. We have
\[
f\circ \pi_n=(f\circ\sigma)\circ\pi_{n+1}
\]
and so $\mathcal F_n\subset\mathcal F_{n+1}$. Since the sigma-algebra $\mathcal F$ on ${\rm Sol}_\sigma(X)$ is the cylinder sigma-algebra 
obtained from $\prod_{n=0}^\infty X$ we have that $\mathcal F=\vee_{n=0}^\infty \mathcal F_n$, where $\vee$ denotes the lattice operation on
sigma-algebras.
\end{proof}

\begin{theorem}
\label{thmmm}
Let $(X,\mathcal B,\sigma, R,h,\lambda)$ be as above, and assume $Rh=h$. Then there exists a unique probability measure $\mathbb P$ defined on the cylinder 
sigma-algebra on the associated solenoid ${\rm Sol}_\sigma(X)$ such that
\begin{equation}
\int_{{\rm Sol}_\sigma(X)}(f_0\pi_0)(\w)(f_1\pi_1)(\w)\cdots (f_n\pi_n)(\w)d\mathbb P(\w)=
\int_{X}(f_0(x)R(f_1R(f_2\cdots R(f_nh))))(x)d\lambda(x)
\end{equation}
for all $n\in\mathbb N_0$ and $f_0,\ldots, f_n\in \mathcal M(X,\mathcal B)$. In the normalized case, we use $h\equiv1$
\end{theorem}
\begin{proof}
 We first remark that, in view of \eqref{sol1},  $\mathbb P$ (if it exists) is uniquely determined by
\begin{equation}
\label{eqn31}
\int_{{\rm Sol}_\sigma}(f\circ \pi_0)(\w)d\mathbb P(\w)=\int_Xf(x)d\lambda(x),
\end{equation}
and
\begin{equation}
\label{eqn3}
\int_{{\rm Sol}_\sigma(X)}(f\circ \pi_n)(\w)d\mathbb P(\w)=\int_XR^n(fh)(x)d\lambda(x).
\end{equation}
For every $n\in\mathbb N_0$, there exists a measure $\mathbb P_n$ on $\mathcal F_n$ such that
\begin{equation}
\label{eqn311}
\int_{{\rm Sol}_\sigma(X)}(f\circ \pi_n)(\w)d\mathbb P_n(\w)=\int_XR^n(fh)(x)d\lambda(x).
\end{equation}
Setting, in \eqref{eqn3}, $n+1$ instead of $n$, and taking into account \eqref{sol1}, we have
\[
\begin{split}
\int_{{\rm Sol}_\sigma(X)}(f\circ \pi_{n})(\w)d\mathbb P_{n+1}(\w)&=\int_{{\rm Sol}_\sigma(X)}(f\circ \pi_{n+1})(\w)d\mathbb P_{n+1}(\w)\\
&=\int_XR^{n+1}(f\circ \sigma)(x)d\lambda(x)\\
&=\int_XR^{n}\left(R(f\circ \sigma)\right)(x)d\lambda(x)\\
&=\int_X\left(R^{n}f\right)(x)d\lambda(x)\\
&\hspace{-3.cm}(\text{by the pull-out property \eqref{roxanna} since $R$ is normalized})\\
&=\int_{{\rm Sol}_\sigma(X)}(f\circ \pi_{n})(\w)d\mathbb P_{n}(\w).
\end{split}
\]

By Kolmogorov's extension theorem (see e.g. \cite{MR33:6659}), the family $(\mathbb P_n)_{n\in\mathbb N}$ extends to a probability measure 
$\mathbb P$ on the cylinder sigma-algebra.
\end{proof}
We define the measure $(\lambda R)$ by
\begin{equation}
\label{lambdaR}
\int_Xf(x)d(\lambda R)(x)=\int_XR(f)(x)d\lambda(x).
\end{equation}

\begin{definition}
{\rm
We say that $\sigma$ is ergodic if
\[
\cap_{n=1}^\infty\sigma^{-n}\left(\mathcal B_X\right)=\left\{\emptyset, X\right\},
\]
modulo sets of $\lambda$-measure zero.}
\end{definition}

The examples of endomorphisms $\sigma$ which we consider here are ergodic.

\begin{theorem}
Let $W$ be a positive measurable function on $(X,\mathcal B)$.
The following are equivalent:\\
$(1)$ $\lambda R<<\lambda$, and
\begin{equation}
\label{lambdaR1}
\frac{d\lambda R}{d\lambda}=W.
\end{equation}
$(2)$ $\mathbb  P\circ\widehat{\sigma}<< \mathbb P$ and
\begin{equation}
\label{murielle}
\frac{d\mathbb P\circ\widehat{\sigma}}{d\mathbb P}=W\circ \pi_0
\end{equation}
\label{elinor!!!}
\end{theorem}

\begin{proof}
Assume that $(1)$ is in force. To prove $(2)$ we will show that \eqref{murielle} holds, or equivalently, that
\begin{equation}
\label{eq12345}
\int_{{\rm Sol}_\sigma(X)}\psi\mathbb P=\int_{{\rm Sol}_\sigma(X)}\left(\psi\circ \widehat{\sigma}\right)(W\circ \pi_0)d\mathbb P
\end{equation}
for all measurable functions $\psi$ on ${\rm Sol}_\sigma(X)$.It suffices to  take $\psi$ of the form $\psi=f\circ\pi_n$ for $n=0,1,\ldots$. We first 
consider the case $n=0$. Let $\w=(x_0,x_1,\ldots)$. Then $\widehat{\sigma}(\w)=(\sigma(x_0),x_0,x_1,\ldots)$, and so
with $\psi=f\circ \pi_0$, the right-hand side of \eqref{eq12345} is equal to
\[
\begin{split}
\int_{{\rm Sol}_\sigma(X)}(f\circ\pi_0\circ\widehat{\sigma})(\w)(W\circ\pi_0)(\w)d\mathbb P(\w)&=
\int_{{\rm Sol}_\sigma(X)}(f\circ{\sigma})(x_0)(W\circ\pi_0)(\w)d\mathbb P(\w)\\
&=\int_{{\rm Sol}_\sigma(X)}(f\circ\sigma\circ\pi_0)(\w)(W\circ\pi_0)(\w)d\mathbb P(\w)\\
&=\int_{{\rm Sol}_\sigma(X)}\left(\left((f\circ\sigma)W\right)\circ\pi_0\right)(\w)d\mathbb P(\w)\\
&=\int_X(f\circ\sigma)(x)W(x)d\lambda(x)\\
&=\int_X\left(R(f\circ\sigma)\right)(x)d\lambda(x)\quad(\text{{\rm using \eqref{lambdaR} and \eqref{lambdaR1}}})\\
&=\int_Xf(x)d\lambda(x)\quad\text{\rm since $R$ is normalized: $R1=1$}\\
&\hspace{2.5cm}\text{{and using the pullout property \eqref{roxanna}}}\\
&=\int_{{\rm Sol}_\sigma(X)}(f\circ\pi_0)(\w)d\mathbb P(\w),
\end{split}
\]
which is the left-hand side of \eqref{eq12345}.\smallskip

We now consider the case $n>0$. The  right-hand side of \eqref{eq12345} is equal to:
\[
\begin{split}
\int_{{\rm Sol}_\sigma(X)}(f\circ\pi_n\circ\widehat{\sigma})(\w)(W\circ\pi_0)(\w)d\mathbb P(\w)&=
\int_{{\rm Sol}_\sigma(X)}f(x_{n-1})(W\circ\pi_0)(\w)d\mathbb P(\w)\\
&=\int_{{\rm Sol}_\sigma(X)}(f\circ\pi_{n-1})(\w)(W\circ\pi_0)(\w)d\mathbb P(\w)\\
&=\int_X\left(R^{n-1}(f)\right)(x)W(x)d\lambda(x)\\
&\quad(\text{\rm where we have used \eqref{eqn3}})\\
&=\int_X\left(R^{n}(f)\right)(x)d\lambda(x)\\
&\quad(\text{\rm by definition of $W$})\\
&=\int_{{\rm Sol}_\sigma(X)}(f\circ\pi_n)(\w)d\mathbb P(\w)\\
&\quad(\text{\rm by \eqref{eqn3}})\\
&=\int_{{\rm Sol}_\sigma(X)}\psi(\w)d\mathbb P(\w),
\end{split}
\]
by definition of $\psi=f\circ\pi_n$.\smallskip

Conversely, we now assume that $\mathbb  P\circ\widehat{\sigma}<< \mathbb P$, with Radon-Nikodym derivative given by
\eqref{murielle}. Let $\psi=f\circ\pi_1$. We have:
\begin{equation}
\label{lysbeth}
\int_{{\rm Sol}_\sigma(X)}(\psi\circ\widehat{\sigma})(\w)(W\circ\pi_0)(\w)d\mathbb P(\w)=
\int_{{\rm Sol}_\sigma(X)}\psi(\w)d\mathbb P(\w).
\end{equation}
Since $\pi\circ\widehat{\sigma}=\pi_0$, this latter equality is equivalent to:
\[
\int_{{\rm Sol}_\sigma(X)}\left((fW)\circ\pi_0\right)(\w)d\mathbb P(\w)=\int_{{\rm Sol}_\sigma(X)}(f\circ\pi_1)(\w)d\mathbb P(\w),
\]
that is
\begin{equation}
\int_X(fW)(x)d\lambda(x)=\int_X(R(f))(x)d\lambda(x)
\label{elinor}
\end{equation}
where we used \eqref{eqn3}. But \eqref{elinor} means that $\lambda R<<\lambda$ with Radon-Nikodym derivative  equal to $W$.
\end{proof}

\begin{remark}
{\rm
It follows from \eqref{eqn3} that the probability distribution $\mu_n$ of the random variable $\pi_n$ is equal to
\[
d\mu_n(x)=W(x)((W\circ\sigma)(x))\cdots ((W\circ\sigma^{n-1})(x))d\lambda(x),
\]
and that
\begin{equation}
\frac{d\mu_{n+1}}{d\mu_n}=(W\circ\sigma^n)(x).
\end{equation}
}
\end{remark}

We will assume that there the Radon-Nikodym derivative $W=\frac{d(\lambda R)}{d\lambda}$ exists, that  is:
\begin{equation}
\int_X R(f)(x)d\lambda(x)=\int_Xf(x)W(x)d\lambda(x),\quad\forall f\,\,\in\, \mathcal M(X,\mathcal B),
\label{voltaire}
\end{equation}
and that, furthermore, the Ruelle operator $R$ in equations \eqref{ruelle1234} and \eqref{trans123} is of
Perron-Frobenius type in the sense that there exists
$h\geq 0$, $h\in(X, \mathcal{B})$ such that
\begin{equation}
Rh=h,
\end{equation}
and normalized to
\begin{equation}
\int_X h(x)d\lambda(x)=1.
\end{equation}
 
When $R$ is not normalized one can replace $R$ with the operator $R^\prime$ defined by
\[
R^\prime f=\frac{R(fh)}{h}.
\]
It satisfies $R^\prime 1=1$. See Remark \ref{utopia}.

\begin{definition}{\rm
We will call $(X,\mathcal B,\sigma, R,h,\lambda)$ a generator for a path space when $\lambda R<<\lambda$ and when $R$ is normalized.}
\end{definition}

As a consequence of \eqref{sol1} we have (recall that $\sigma$ is not one-to-one in general):
\begin{proposition}
The distributions $\mu_{k}$ and $\mu_{k+1}$ are related by
\begin{equation}
\label{sol123}
\mu_{k+1}\circ \sigma^{-1}=\mu_k,
\end{equation}
meaning that
\begin{equation}
\label{regine}
\mu_{k+1}\left(\sigma^{-1}(B)\right)=\mu_k(B),\quad \forall B\in\mathcal B.
\end{equation}
\end{proposition}

\begin{proof}
By definition of $\mu_k$ we have:
\[
\int_{{\rm Sol}_\sigma(X)}(f\circ\pi_k)(\w)d\mathbb P(\w)=\int_Xf(x)d\mu_k(x).
\]
Hence,
\[
\begin{split}
\int_X(f\circ \sigma)(x)d\mu_{k+1}(x)&=\int_{{\rm Sol}_\sigma(X)}((f\circ\sigma)\circ\pi_{k+1})(\w)d\mathbb P(\w)\\
&=\int_{{\rm Sol}_\sigma(X)}(f\circ\pi_k)(\w)d\mathbb P(\w)\quad (\text{\rm using \eqref{sol1}})\\
&=\int_Xf(x)d\mu_k(x).
\end{split}
\]
It suffices to take $f(x)=\chi_B(x)$ to obtain \eqref{regine}.
\end{proof}
\begin{remark}
{\rm
\eqref{regine} is independent of the given probability measure on the cylinder sigma-algebra.
}
\end{remark}
\subsection{The multiresolution associated with a solenoid}
We begin with a table relative to the wavelet realization by unitary operators; the third column, related to the classical 
$\mathbf L_2(\mathbb R,dx)$ wavelets is elaborated upon in Section \ref{unitcircle}.

\begin{definition}
\label{newdef123}
{\rm
Let $\mathcal H$ be a Hilbert space, let $U\,:\,\mathcal H\,\longrightarrow\,\mathcal H$ be a unitary operator and let $(\mathcal H_n)_{n\in
\mathbb Z}$ be an indexed family of closed subspaces such that:\\
$(i)$ $\mathcal H_{n+1}\subset \mathcal H_n$, $n\in\mathbb Z$,\\
$(ii)$ $U^{-k}\mathcal H_0=\mathcal H_k$, $k\in\mathbb Z$,\\
$(iii)$ $\bigwedge_{k\in\mathbb Z}\mathcal H_k$ is at most one dimensional,\\ 
and\\
$(iv)$ $\bigvee_{k\in\mathbb Z}\mathcal H_k=\mathcal H$.\\
Here $\bigwedge$ and $\bigvee$ refer to the lattice operations applied to closed subspaces in $\mathcal H$. \\
When $(\mathcal H, U, (\mathcal H_n)_{n\in\mathbb Z})$ satisfy $(i)$-$(iv)$, then we say that it is a multiresolution
(or multi-scale resolution), and that $U$ is the associated scaling operator.}
\end{definition}

\begin{remark}{\rm
Let $U$ be a unitary operator which is part of a multiresolution, then it can be shown that the spectrum of $U$ must be as follows: Except for the point $\lambda = 1$  occurring with at most multiplicity one, the spectrum of $U$ must be absolutely continuous with uniform multiplicity infinity. This is an application of ideas of Wold, Lax-Phillips, and Stone-von Neumann; see 
\cite{MR1037774,MR1171011}. See also Remark \ref{cov} below.
}
\end{remark}
We shall outline below a number of examples of multiresolutions, in wavelet theory and in dynamics more generally. This will make use of the
theory we already developed in Section 2 above.
\begin{table}[h!]
\caption{Wavelets  realization by unitary operators}
\begin{tabular}{|c|c|c|c|}
\toprule
& & &\\
\text{The case} &$\mathbf L_2(\mathbb R,dx)$ &Fourier transform &General (solenoid)\\
& &of $\mathbf L_2(\mathbb R,dx)$ & $\mathbf L_2({\rm Sol}_\sigma,\mathcal C,\mathbb P)$\\
\hline
& & &\\
\text{The unitary}& & &\hspace{-1.4cm}$(U\psi)(\w)=$\\
\text{operator}&$(Ug)(x)=\frac{1}{\sqrt{N}}g(x/N)$ & $(U\gamma)(t)=\sqrt{N}\gamma(Nt)$& $=(\psi\circ\widehat{\sigma})
(m_0\circ\pi_0)$\\
& & &\\
\hline
& & &\\
\text{\hspace{-1.6cm} $(1)$ Map onto}&$K\xi=\sum_{k\in\mathbb Z}\xi_k\varphi(x-k)$ &$(Kf)(t)=f(t)\widehat{\varphi_0}(t)$ & $V_{\pi_0}f=f\circ \pi_0$\\
\text{\hspace{-4mm}the zero resolution}&\hspace{-3.5mm}belongs to $\mathbf L_2(\mathbb R,dx)$ & \hspace{-3.5mm}belongs to $\mathbf L_2(\mathbb R,dx)$ 
&belongs to $\mathbf L_2({\rm Sol}_\sigma,\mathbb P)$\\
\text{\hspace{-2.cm}subspace.}&\hspace{-11.5mm}\text{for  }$(\xi_n)\in\ell_2(\mathbb Z)$. &
\hspace{-8mm}\text{for $f\in\mathbf L_2(\mathbb T)$.} &\hspace{-7.6mm}\text{for  } $f\in\mathbf L_2(X,\lambda)$.\\
& & &\\
\hline
& & &\\
\text{$(2)$ Average operator.}&$(S\xi)_j=\sum_{k\in\mathbb Z}\xi_ka_{Nk-j}$ &$(Sf)(t)=m_0(t)f(Nt)$ &$Sf=m_0\cdot f\circ \sigma$\\
& & &\\
\hline
\text{\hspace{-5.mm}Level zero resolution}& & &\\
\text{\hspace{-8mm}Invariant subspace}&$\xi\in\ell_2(\mathbb Z)\simeq \mathcal H_0$ & $f\in\mathbf L_2(\mathbb T)\simeq \mathcal H_0
$&$f\in\mathbf L_2(X,\lambda)\simeq \mathcal H_0^{{\rm Sol}_\sigma}$\\
\text{for $S$ (using $(1)$ above)} & & &\\
\bottomrule
\end{tabular}
\label{rubicon2}
\end{table}
\begin{remark}{\rm
In the above table, in the second column, the map $m_0$ is a continuous function $m_0(z)$ on the unit circle, and the coefficients $a_n$ 
are the Fourier coefficients of $|m_0(e^{it})|^2$. In the third column, we are in the special case where $W=|m_0|^2$, then $W$ satisfies 
the conditions of Theorem \ref{elinor}.}
\end{remark}

Recall from Lemma \ref{wzero} that if $\mu_0$ and $\mu_1$ are equivalent. Then, the set where $W$ vanishes has measure zero.

\begin{proposition}
Assume that $\mu_0$ and $\mu_1$ are equivalent, Then, the map
\begin{equation}
\label{isabelle}
U\psi=\sqrt{W\circ \pi_0}\left(\psi\circ\widehat{\sigma}\right)
\end{equation}
is unitary from $\mathbf L_2({\rm Sol}_\sigma(X),\mathbb P)$ onto itself, and its inverse is given by
\begin{equation}
U^{-1}\psi=\frac{1}{\sqrt{W\circ\pi_1}}\psi\circ\widehat{\sigma}^{-1},
\end{equation}
where $\widehat{\sigma}^{-1}$ is given by \eqref{tau123}. We now check that $UU^{-1}\psi=\psi$. We have
\[
\begin{split}
UU^{-1}\psi&=\sqrt{W\circ\pi_0}\left(U^{-1}\psi\right)\circ\widehat{\sigma}^{-1}\\
&=\sqrt{W\circ\pi_0}\frac{1}{\sqrt{W\circ\pi_1\circ\widehat{\sigma}}}\psi\circ\widehat{\sigma}\circ\widehat{\sigma}^{-1}\\
&=\psi
\end{split}
\]
since $\pi_1\circ \widehat{\sigma}=\pi_0$. The proof that $U^{-1}U\psi=\psi$ is similar, and omitted.
\end{proposition}

\begin{proof}
The first claim  is a consequence of \eqref{lysbeth} with $|\psi^2|$ instead of $\psi$.
\end{proof}

\begin{definition}
{\rm
Let 
\begin{equation}
\mathcal H_0=\left\{f\circ \pi_0\,\,|\,\,f\in\mathbf L_2(X,hd\lambda)\right\},
\end{equation}
the resolution subspace.
The family $\mathcal H_n=U^{-n}\mathcal H_0$, $n\in\mathbb Z$, is called the multiresolution associated with the solenoid, and will be denoted
by ${\rm MR}_{\sigma}$.}
\end{definition}

Note that
\begin{equation}
\label{odeon}
\int_{\mathbf L_2(\mathbb P)}|f\circ \pi_0|^2d\mathbb P=\int_X|f(x)|^2h(x)d\lambda(x).
\end{equation}

\begin{proposition}
Let $f\in L^\infty(X)$. The multiplication map
\begin{equation}
M_{f\circ \pi_0}
\end{equation}
sends $\mathbf L_2({\rm Sol}_\sigma,\mathbb P)$ into itself, and $\mathcal H_n$ into itself for all $n$. We have
\[
\begin{array}{ccc}
\mathbf L_2(\Omega,\mathcal F,\mathbb P)&\xrightarrow{\hspace*{0.5cm}{U}\hspace*{.5cm}}
%
&\mathbf L_2(\Omega,\mathcal F,\mathbb P)\\ 
{M_{f\circ \pi_0}}\xdownarrow{0.6cm}
& &\xdownarrow{0.6cm} {M_{f\circ\sigma\circ\pi_0}}\\
\mathbf L_2(\Omega,\mathcal F,\mathbb P)&\xrightarrow{\hspace*{0.5cm}{U}\hspace*{.5cm}}
%
&\mathbf L_2(\Omega,\mathcal F,\mathbb P)\\ 
\end{array}.\]
and the following covariance relation holds (see also Remarks \ref{cov} and \ref{cov2})
\begin{equation}
\label{jardin-du-luxembourg}
UM_{f\circ\pi_0}U^{-1}=M_{f\circ \sigma\circ \pi_0}.
\end{equation}
The map $f\mapsto M_f$ defines a representation of $L^\infty(X)$ by bounded operators on $\mathbf L_2({\rm Sol}_\sigma,\mathbb P)$.
\end{proposition}

\begin{proof} For $\psi\in \mathcal M(\Omega)$ we have
\[
\begin{split}
UM_{f\circ \pi_0 }\psi&=\sqrt{W\circ\pi_0}(f\circ\pi_0\circ\widehat{\sigma})(\psi\circ\widehat{\sigma})\\
&=\sqrt{W\circ\pi_0}(f\circ\sigma\circ\pi_0)(\psi\circ\widehat{\sigma})\quad (\text{\rm where we use $\pi_0\circ\widehat{\sigma}=\sigma\circ\pi_0$;
see \eqref{corinne}})\\
&=M_{f\circ\sigma\circ\pi_0}U\psi.
\end{split}
\]
\end{proof}

\begin{remark}
\label{cov}
{\rm
Note that \eqref{jardin-du-luxembourg} is an instance of a covariance relation: It states that the representation $M$ is unitarily equivalent to the representation obtained from it by substitution with the endomorphism $\sigma$. As a result, the projection valued measure determining $M$ will satisfy the analogous covariance. This is outlined in \eqref{UEUE} below. For the convenience of the reader, let us give the
following analogy: Consider the two canonical variables $P$ and $Q$ in the canonical commutation relation from quantum mechanics; in the Weyl exponentiated form. If $E_Q$ denotes the projection valued spectral measure of $Q$, then the unitary one-parameter group $U(t)$, generated by $P$, satisfies a covariance in the form
\[
U(t) E_Q(B) U(-t)  =  E_Q(B + t),
\]
all for all Borel sets $B$, and all $t\in\mathbb R$.
Here we use the word “covariance” in the same general context, but now for endomorphisms, also now instead for a single unitary operator. Many covariance relations have solutions that are unique up to unitary equivalence, for example the canonical $P-Q$ relation  does; this is a form of the Stone-von Neumann uniqueness theorem. See \cite{MR555264,MR1037774,MR1171011}.}
\end{remark}

\begin{remark}{\rm
The commutative von Neumann algebra $\mathcal M_{\pi_0}$ of the multiplication operators $M_{f\circ \pi_0}$ with $f\in\mathbf L^\infty(X,\mathcal B_X)$ has the 
spectral representation (see Section \ref{stonesec} and equation \eqref{condi12345})
\[
M_{f\circ \pi_0}=\int_Xf(x)\chi_{\left\{\pi_0\in dx\right\}}
\]
where $\chi_{\left\{\pi_0\in dx\right\}}$ (also denoted by $\mathscr E_{\pi_0}(dx)$ is the  projection-valued measure given by \eqref{EA} and 
arising from the Stone theorem applied to $\mathcal M_{\pi_0}$; see \cite{Nelson_flows}.\smallskip

Define 
\[
\mathscr E_{\pi_0}^{(\sigma)}(\w)=M_{\chi_{\left\{\pi_0\in\sigma^{-1}(L)\right\}}}
\]
As in \eqref{jardin-du-luxembourg} we arrive at the following selfsimilarity property for $\mathscr E_{\pi_0}$ with $U$ given by \eqref{isabelle}:
\begin{equation}
U\mathscr E_{\pi_0}(L)U^{-1}=\mathscr E_{\pi_0}\left(\sigma^{-1}(L)\right),\quad \forall L\in\mathcal B_X,
\label{UEUE}
\end{equation}
which we also rewrite as $U\mathscr E_{\pi_0}U^{-1}=\mathscr E_{\pi_0}^{(\sigma)}$.
}
\label{cov2}
\end{remark}

We now give another interpretation of the resolution subspace $\mathcal H_n$. For $\mathcal F_n$, see Definition \ref{fn}.

\begin{proposition}
We have:
\begin{equation}
\label{matilde}
\mathcal H_n=\mathbf L_2({\rm Sol}_\sigma,\mathcal F_n,\mathbb P)=\mathbb E_{\mathcal F_n}(\mathbf L_2({\rm Sol}_\sigma,\mathcal F,\mathbb P)).
\end{equation}
\end{proposition}

\begin{proof}
The proof follows from Corollary \ref{corol211}.
\end{proof}

Equation \eqref{martingale} below is a generalization of the classical notion of martingale. 

\begin{proposition}
Assume $R$ normalized, {\rm i.e.} $R1=1$. Then
\begin{equation}
\label{martingale}
\mathbb E\left(f\circ\pi_{n+1}\big|\mathcal F_n\right)
=R(f)\circ \pi_n,\quad n=0,\ldots
\end{equation}
\end{proposition}

\begin{proof}
This follows from Lemma \ref{lemma24}.
\end{proof}
\begin{definition}
{\rm 
The sequence $(T_n)_{n\in\mathbb N_0}$ of random variables from the probability space $(\Omega,\mathcal F,\mathbb P)$ into the measurable space 
$(X,\mathcal B)$ is called a Markov chain if
\begin{equation}
Pr(T_{n+1}\in B\big| T_0,\ldots, T_n)=Pr(T_{n+1}\in B\big| T_n)=(R(\chi_B))\circ T_n
\end{equation}
}
\end{definition}
\begin{proposition}
We have
\begin{equation}
\label{martingale-sept-78}
\mathbb E\left(f\circ\pi_{n+1}\big|\mathcal F_0,\ldots, \mathcal F_n\right)=\mathbb E\left(f\circ\pi_{n+1}\big|\mathcal F_n\right),
\quad n=0,\ldots
\end{equation}
\label{prop323}
\end{proposition}
\begin{proof}
The result follows from Theorem \ref{thhmc}.
\end{proof}
\begin{theorem}
Assume that $R$ is normalized, and let $(\pi_n)_{n\in\mathbb N_0}$ be the stochastic process on ${{\rm Sol}_\sigma(X)}$ defined by the coordinates.
Then:
\begin{equation}
\label{martingale1}
\mathbb E\left(f\circ\pi_{n+k}\big|\mathcal F_n\right)=R^k(f)\circ \pi_n,\quad n=0,\ldots,\quad k=1,2,\ldots
\end{equation}
\end{theorem}

\begin{proof}
The proof uses the chain rule for conditional expectation and induction. It is enough to consider the case $k=2$. We then have:
\[
\begin{split}
\mathbb E\left(f\circ\pi_{n+2}\big|\mathcal F_n\right)&=\mathbb E\left(\mathbb E\left(f\circ\pi_{n+2}\big|\mathcal F_{n+1}\right)\big|\mathcal F_n\right)\\
&=\mathbb E\left(R(f)\circ\pi_{n+1}\big|\mathcal F_n\right)\\
&=R^2(f)\circ\pi_n,
\end{split}
\]
where we used twice \eqref{martingale}.
\end{proof}
\subsection{Conditional expectations associated with a solenoid}
\label{conditional}
\begin{proposition}
The map $Vf=f\circ\pi_0$ is an isometry from $\mathbf L_2(X,hd\lambda)$ into $\mathbf L_2({\rm Sol}_\sigma,\mathbb P)$.
\end{proposition}

\begin{proof}
This is a corollary of Lemma \ref{lemma22}.
\end{proof}

\begin{definition}{\rm
The projection $\mathbb E_0=VV^*$ in $\mathbf L_2({\rm Sol}_\sigma,\mathbb P)$ is called the conditional expectation onto 
$\mathcal F_0$ 
of the multiresolution $({\rm Sol}_\sigma,\mathbb P)$.}
\end{definition}
\subsection{A general setting and an inverse problem}

We now present a general setting, which includes the preceding analysis. We start from a probability space $(\Omega,\mathcal F,\mathbb P)$, and a 
measurable space $(X,\mathcal B)$. We assume given a sequence of random variables $(T_n)_{n\in\mathbb N_0}$ from $\Omega$ to $X$, and an endomorphism
$\sigma$ from $X$ into $X$. We assume that
\begin{equation}
\sigma \circ T_{n+1}=T_n,\quad n=0,1,\ldots,
\end{equation}
\label{totalrecall}
or, equivalently,
\begin{equation}
\forall \,\w\in\Omega,\,\, T_{n+1}(\w)\in\sigma^{-1}(T_n(\w)).
\end{equation}
The map
\begin{equation}
\widetilde{T}(\w)=(T_0(\w),T_1(\w),\ldots)
\end{equation}
is measurable from $\Omega$ into ${\rm Sol}_\sigma(X)$. It induces a probability measure $\mathbb P^*$ on the cylinder sigma-algebra of $X^{\mathbb N}$ via 
the formula
\begin{equation}
\mathbb P^*(A)=\mathbb P(\widetilde{T}^{-1}(A)).
\end{equation}
The sequence $(T_n)_{n\in\mathbb N_0}$ generates a family of sigma-algebras, namely
\begin{equation}
\mathcal F_n=\left\{T_n^{-1}(A)\,;\, A\in\mathcal B\right\}.
\end{equation}
In  view of \eqref{totalrecall} we have $\mathcal F_n\subset\mathcal F_{n+1}$.\smallskip

We now recall a technical lemma, to be used in the proof of Theorem \ref{thm123}

\begin{lemma}
Let $(\Omega,\mathcal F)$ and $(X,\mathcal B)$ be two measurable spaces and let $T$ be a map from $\Omega$ into $X$. Let 
\[
\mathcal F_T=\left\{T^{-1}(B)\,;\, B\in\mathcal B)\right\}.
\]
Then a real valued function $\psi$ defined on $\Omega$ is $\mathcal F_T$-measurable if and only if it can written in the form 
\[
\psi=f\circ T
\]
for a uniquely defined $\mathcal B$-measurable function $f$.
\label{bastille}
\end{lemma}

\begin{theorem}
\label{thm123}
There exists a positive operator defined on the space of measurable functions from $X$ to $\mathbb R$ such that
\begin{equation}
\label{yorba-linda}
\mathbb E\left(f\circ T_{n+1}\big|\mathcal F_n\right)=R(f)\circ T_n.
\end{equation}
\end{theorem}
\begin{proof}
The existence of $R$ follows from Lemma \ref{bastille}, and the positivity of $R$ follows from the  fact that a conditional expectation is an orthogonal
projection.
\end{proof}

\begin{corol}
\begin{eqnarray}
\label{eqnorm1234}
R(1)&=1,\\
R((f\circ \sigma)g)&=&fR(g).
\label{eqnorm12341234}
\end{eqnarray}
\end{corol}

\begin{proof}
The first equation follows from setting $n=0$ and $f\equiv 1$ in \eqref{yorba-linda}. The second equation is proved as follows.
We have
\[
\begin{split}
\mathbb E\left(((f\circ \sigma)g)\circ T_1\,\big|\,\mathcal F_0\right)&=\mathbb E\left((f\circ \sigma\circ T_1)(g\circ T_1)\,\big|\,\mathcal F_0\right)\\
&=\mathbb E\left((f\circ T_0)(g\circ T_1)\,\big|\,\mathcal F_0\right)\\
&=(f\circ T_0)\mathbb E\left((g\circ T_1)\,\big|\,\mathcal F_0\right)\\
&=(f\circ T_0)(R(g)\circ T_0)\\
&=(fR(g))\circ T_0.
\end{split}
\]
\end{proof}

\section{Examples and applications: Transfer operators and Markov moves}
\label{sec4}
\setcounter{equation}{0}
While in the abstract, as we showed, Markov chains are derived from positive operators $R$, acting on functions on a fixed measure space 
$(X, \mathcal B_X)$. Starting with a choice of $R$ (the transfer operator), we then build a Markov chain $T_0, T_1, T_2,\ldots$, with 
these random variables (r.v) realized in a suitable probability space $(\Omega, \mathcal F, \mathbb P)$, and each r.v.  taking values in $X$, 
measurable of course with respect to the respective sigma algebras, $\mathcal F$ on $\Omega$, and $\mathcal B_X$ on $X$. Conversely, 
every Markov chain is determined by its transfer operator.\smallskip
    
The purpose of the examples below is to put this correspondence into more practical terms. The range of the examples we give will cover 
$(i)$ iterated function systems (IFS), $(ii)$ wavelet multiresolution constructions, and $(iii)$  IFSs with random “control.”\smallskip

An IFS on a fixed measure space $(X, \mathcal B_X)$ is a system of mappings $\tau_i$, each acting in $(X, \mathcal B_X)$, and each assigned 
a probability, say $p_i$   which may or may not be a function of $x$. For standard IFSs it is not, but for wavelet constructions it is. 
In the latter, the functions $p_i(x)$ reflect the multi-band filters making up the wavelet algorithm. Moreover, the sets $\tau_i(X)$  
partition $X$, but they may have overlap, or not. The Markov chains for the non-overlapping IFSs are simpler.\smallskip

Returning to the general case, we now briefly sketch the idea behind Markov moves in IFSs with random control in a bit more detail. 
The examples below will supply hands-on cases, serving to illustrate the general idea.\smallskip

The Markov move:  Starting with a point $x$ in $X$, the Markov move to the next point is in two steps, as follows, the combined two steps 
describing the move from $T_0$ to $T_1$, and more generally from $T_n$ to $T_{n+1}$. The initial point $x$ will first move to one of the sets 
$\tau_i(X)$  with probability  $p_i$, and once there, it will “choose” a definite position (within $\tau_i(X)$), and this second move 
will be prescribed by a fixed law (a given probability distribution); for example, the law could be the uniform distribution, or 
something different. However, for Markov chains, the law is the same in the move from $T_n$ to $T_{n+1}$, for all $n$.

\subsection{First examples}
\begin{example}
{\rm
In the first example, $X=[0,1]$ and $\sigma(x)=4x(1-x)$, called the logistic map. von Neumann and Ulam proved that an invariant measure is
\[
d\mu(x)=\frac{dx}{\pi\sqrt{x(1-x)}},
\]
the Beta $B(\frac{1}{2},\frac{1}{2})$ distribution, i.e. $\mu\circ\sigma^{-1}=\mu$. See \cite[pp. 87-91]{MR1974383}. 
The corresponding transfer operator is
\begin{equation}
(Rf)(x)=\frac{1}{2}\left(f\left(\frac{1+\sqrt{1-x}}{2}\right)+f\left(\frac{1-\sqrt{1-x}}{2}\right)\right)
\end{equation}

We note that
\[
\mu R\not=\mu.
\]
}
\end{example}

We now turn to an example of a transfer operator $R_F\,:\, X\times Y\,\longrightarrow\, X$
\[
(R_Ff)(x)=\int_Yf(F(x,y))d\nu(y)
\]
in which 
\begin{equation}
\label{mumumu}
\mu R_F=\mu
\end{equation}
for the $B(\frac{1}{2},\frac{1}{2})$ law $\mu$. As a consequence of \eqref{mumumu}, we have that the corresponding probability measure 
$\mathbb P$ in $(\prod_{n=0}^\infty X,\mathcal F,\mathbb P)$ will be shift-invariant.

\begin{example}
{\rm 
We take $X=(0,1)$. The endomorphism $\sigma$ will depend on a parameter
$u\in (0,1)$, and is defined as follows. Set

\begin{eqnarray}
\tau^{(u)}_0(x)&=&ux,\\
\tau^{(u)}_1(x)&=&u+(1-u)x.
\end{eqnarray}
Then,
\begin{equation}
\sigma^{(u)}(x)=\begin{cases}\,\frac{x}{u},\quad\hspace{1.8cm} 0<x\le u,\\
\,-\frac{u}{1-u}+\frac{x}{1-u},\quad u<x<1.
\end{cases}
\end{equation}
Then, 
\[
\sigma^{(u)}\circ\tau_i^{(u)}(x)=x,\quad {\rm for}\,\,\,i=1,2,\quad{\rm and}\quad x\in(0,1).
\]
Then,
\begin{equation}
R^{(u)}f(x)=\frac{1}{2}\left(f(\tau_0^{(u)})(x)+f(\tau_1^{(u)})(x)\right).
\end{equation}
Let $\lambda$ be the Lebesgue measure on $[0,1]$. Then
\begin{equation}
d(\lambda R^{(u)})=W^{(u)}d\lambda,\quad u\in(0,1)
\end{equation}
with
\[
W^{(u)}(x)=\begin{cases}\, \frac{1}{2u},\quad\,\,\,\,\,\,\,\, 0\le x<u,\\
\,\frac{1}{2(1-u)},\quad u\le x<1.
\end{cases}
\]
Note that $W^{(u)}(x)\equiv x$ if and only if $u=\frac{1}{2}$.
For every $u\in(0,1)$ we have a quasi-invariant measure $\mathbb P^{(u)}$ such that
\[
\frac{\mathbb P^{(u)}\circ\widehat{\sigma^{(u)}}}{d\mathbb P^{(u)}}=W^{(u)}\circ \pi_0.
\]
Let $Y=\left\{0,1\right\}\times (0,1)$ and $d\nu=p_1\times p_2$ be the product measure with $p_1(0)=p_1(1)=\frac{1}{2}$ and $p_2$ the uniform 
probability distribution on $(0,1)$. Let furthermore
\[
F(x,(i,u))=\begin{cases}\,ux,\quad\quad\hspace{1.35cm}{\rm if}\quad i=0,\\
\, (1-u)x+u,\quad{\rm if}\quad i=1,
\end{cases}  
\]
and 
\begin{equation}
\label{elinor123123}
\begin{split}
(R_Ff)(x)&=\int_Yf(F(x,(i,u)))d\nu(i,u)\\
&=\frac{1}{2}\int_0^1\left(f(ux)+f((1-u)x+u)\right)du\\
&=\frac{1}{2}\left(\frac{1}{x}\int_0^xf(t)dt+\frac{1}{1-x}\int_x^1f(t)dt\right).
\end{split}
\end{equation}
}
\end{example}

Now we show that the transfer operator which we just introduced   has an invariant measure with absolutely continuous density.

\begin{proposition}
Let $R_F$ denote the transfer operator defined in \eqref{elinor123123}, and set 
\begin{equation}
\label{123mu123}
d\mu(x)=\frac{dx}{\pi\sqrt{x(1-x)}}, \qquad x\in(0,1).
\end{equation}
We then have
\[
\mu R_F=\mu,
\]
that is,
\begin{equation}
\label{liza}
\int_0^1(R_Ff)(x)d\mu(x)=\int_0^1f(x)dx,\quad \forall\,f\,\in\,\mathcal M((0,1),\mathcal B).
\end{equation}
\end{proposition}

\begin{proof}
For $d\mu(x)=G(x)dx$ to satisfy \eqref{liza} we must have
\[
G(y)=\frac{1}{2}\left(\int_y^1\frac{G(x)}{x}dx+\int_0^y\frac{G(x)}{1-x}dx\right).
\]
Hence
\[
\frac{G^\prime(y)}{G(y)}=\frac{1}{2}\left(-\frac{1}{y}+\frac{1}{1-y}\right),
\]
and hence the result.
\end{proof}

\begin{definition}
{\rm
We define the backward shift $s$ on sequences of $\prod_{n=0}^\infty X$ by
\begin{equation}
s(x_0,x_1,x_2,\ldots)=(x_1,x_2,\ldots).
\end{equation}
}
\end{definition}

\begin{proposition}
In the setting of Theorem \ref{loudmila},  let $\mu$ be an invariant measure for the transfer operator, and let $\pi_0$ be endowed  with $\mu$ 
as probability law. Then the corresponding probability measure is shift-invariant: 
\[
\mathbb P\circ s^{-1}=\mathbb P.
\]
\end{proposition}

\begin{proof}
$\mathbb P$ is built from the Kolmogorov construction by
\begin{equation}
\int_{{\rm Sol}_\sigma(X)}(f_0\pi_0)(\w)(f_1\pi_1)(\w)\cdots (f_n\pi_n)(\w)d\mathbb P(\w)=
\int_{X}(f_0(x)R(f_1R(f_2\cdots R(f_nh))))(x)d\mu(x).
\end{equation}
\end{proof}
\subsection{Cases where $\sigma$ is not onto}
When the endomorphism $\sigma$ is not onto, the solenoid satisfies
\[
{\rm Sol}_\sigma(X)\subset \prod_{n=1}^\infty X_\infty^{(\sigma)},\quad{\rm where}\quad X_\infty^{(\sigma)}\stackrel{\rm def.}{=}
\cap_{n=1}^\infty\sigma^n(X),
\]
and the latter can be a very small set, as we now illustrate.

\begin{example}
{\rm
Take $X=[0,1]$ and $\sigma(x)=2x(1-x)$. Then 
\[
\sigma(X)=[0,\frac{1}{2}]\quad{\rm  and}\quad X_\infty^{(\sigma)}=\left\{0,\frac{1}{2}\right\}.
\]
The solenoid consists of the two points
\[
(0,0,\ldots)\quad{\rm and}\quad (\frac{1}{2},\frac{1}{2},\ldots).
\]
}
\end{example}

\begin{example}
{\rm
This example is from complex dynamics. We take $X=\mathbb C$ and for a pre-assigned $c\in \mathbb C$,
\[
\sigma_c(z)=z^2+c.
\]
Then $ X_\infty^{(\sigma)}$ is the Julia set, see \cite{MR1721240}.
}
\end{example}
\subsection{Solenoids associated with the unit circle}
\label{unitcircle}

 In the period since the mid 1990ties, the term “wavelet” has come to have a broader meaning: From referring to systems of bases in 
$\mathbf L_2(\mathbb R)$ with dyadic scale symmetry, “wavelet” now typically refers to finite systems of functions on a suitable measure space that can be used in order to construct either an orthonormal basis, or frame basis by means of operators connected to algebraic and geometric information involving a notion of “scaling function.” The latter often in the form of a probability measure on a solenoid-measure space. In the case of fractals, there are natural choices of finite systems of functions yielding very well-behaved orthonormal bases, and thus giving direct information about the topological structure of the particular fractal involved. Our framework below makes use of solenoids (from endomorphisms) in order to offer an even more inclusive framework for wavelet bases and multiresolutions. Background references for the present section include 
\cite{MR2888226,zbMATH03176105,BrJo02a,MR3275999,MR2268116,MR2391805,MR1837681,MR2821778,MR3204025}.

\subsubsection{Definition}
Starting from a continuous function $m_0(z)$ on the unit circle $\mathbb T$ and $N\in\left\{2,3,\ldots\right\}$ 
one can construct (at least) two representations of the algebra of operators  generated by
two operators $T,U$ such that $U$ is unitary and $UTU^{-1}=T^N$ (such an algebra is an algebra generated by a group of the kind studied
in \cite{zbMATH03176105} by Baumslag and Solitar). To be more precise 
let $R=R_{m_0}$ denote the corresponding Ruelle operator:
\begin{equation}
\label{ruelle1234}
(Rf)(z)=\frac{1}{N}\sum_{w^N=z}|m_0(w)|^2f(w).
\end{equation}
When $R 1= 1$ the infinite product $\prod_{u=1}^\infty \frac{m_0(e^{\frac{2\pi i t}{N^u}})}{\sqrt{N}}$ belongs to 
$\mathbf L_2(\mathbb R)$, and is the Fourier transform of the scaling function $\varphi_0$. The space
\begin{equation}
\label{llg1973}
\mathcal H_0=\left\{\widehat{\varphi_0}(t)f(t)\,;\, f\,\,\text{{\rm function on }}\mathbb R/\mathbb Z\,\,\text{\rm measurable and } f(t)=f(t+1)\right\}
\end{equation}
is the $0$-resolution subspace of the multiresolution
\begin{equation}
\label{llg1973-1975}
\mathcal H_k=\left\{2^{-k/2}\widehat{\varphi_0}(t/2^k)f(t)\,;\, f\,\,\text{{\rm measurable function on }}\mathbb R/\mathbb Z\,\,\text{\rm i.e.} 
f(t)=f(t+1)\right\},\quad k\in\mathbb Z.
\end{equation}

One defines a representation $\rho$ of  $\mathbf L_\infty(\mathbb T)$ into $\mathbf B(\mathbf L_2(\mathbb R))$ as follows: If 
$f\in\mathbf L_\infty(\mathbb T)$ with associated Fourier series $f(e^{it})=\sum_{n\in\mathbb Z}\widehat{f}(n)e^{int}$, 
one sets

\begin{equation}
\label{pi}
\rho(f)(g)=\sum_{n\in\mathbb Z} \widehat{f}(n)g(x-n),\quad g\in\mathbf L_2(\mathbb R).
\end{equation}

In this paper we remove the $\mathbf L_2(\mathbb R,dx)$ requirement (which we assumed in \cite{ajl1papers,ajl1paper}) from the wavelet setting. 
Now wavelet multiresolutions may be viewed as a special case of a probability space 
multiresolution. In the latter, the resolution subspaces will be specified by a system of conditional expectations.
In the classical wavelet application, the solenoid becomes the real line, realized as a dense curve in ${\rm Sol}_\sigma(X)$,
and the solenoid measure $P$ becomes Lebesgue measure.\smallskip

We now consider the special case where $X$ is equal to the unit circle $\mathbb R/\mathbb Z=\mathbb T$ and $\sigma(z)=z^N$. When
using the notation $z=e^{2\pi it}$ we have $\sigma(t)=Nt\,\, ({\rm mod}\, 1)$.\\

The solenoid $G_N$ is a compact group, included in $\prod_{k=0}^\infty \mathbb T$,
and consists of the sequences $z=(z_0,z_1,z_2,\ldots)\in\prod_{k=0}^\infty \mathbb T$ such that
\[
z_{k+1}^N=z_k,\quad k=1,2,\ldots
\]
See \cite{MR1837681}.
We define
\begin{equation}
\sigma(z_0,z_1,z_2,\ldots)=(z_0^N,z_1^N,z_2^N,\ldots)=(z_0^N,z_0,z_1,\ldots)
\end{equation}
and
\begin{equation}
\tau(z_0,z_1,z_2,\ldots)=(z_1,z_2,\ldots)
\end{equation}
We have
\[
\sigma\circ \tau=\tau\circ\sigma=I.
\]

It is the dual of the discrete group $\mathbb Z[1/N]$, with characters $\chi\left( \frac{\ell}{N^k}\right)$
given by
\begin{equation}
\label{eqduality}
\langle \chi\left(\frac{\ell}{N^k}\right),z\rangle=z_k^\ell,\quad k,\ell=0,1,\ldots
\end{equation}
See \cite{MR3441734,BrJo02a}. Note that \eqref{eqduality} is well defined since
\[
\langle \chi\left(\frac{N\ell}{N^{k+1}}\right),z\rangle=z_{k+1}^{N\ell}=(z_{k+1}^N)^\ell=z_k^\ell=\langle\chi\left( \frac{\ell}{N^k}\right),z\rangle.
\]
\subsubsection{Ruelle operators and wavelets}
We use the term Ruelle operator consistent with \cite{MR3393694,MR1793194,BrJo02a,MR2129258} to indicate a transfer operator which governs branching 
in a number of different context.
Every filter in the family we have can be realized as a wavelet filter on the solenoid. Fix a low-pass filter $m_0$ with the usual properties,
and define \\

Two cases occur: When the function identically equal to $1$ (denoted in this paper by $\mathbf 1$) is an eigenvalue of $R$ with eigenvalue $1$, that is,
\[
\frac{1}{N}\sum_{\substack{w\in\mathbb T\\w^N=z}}|m_0(w)|^2\equiv 1,
\]
one can construct $\varphi_0$ and use the space $\mathbf L_2(\mathbb R,dx)$.
We study the representations of the algebra generated by $(U,T)$ such that
\begin{equation}
UTU^{-1}=T^N
\end{equation}

We take
\[
U\left(\sum  \xi_n\varphi_0(\cdot-n)\right)=m_0(z)f(z^N)\quad\text{with}\quad f(z)=\sum \xi_nz^n.
\]
Thus we have a slanted Toeplitz matrix
\[
(S\xi)_n=\sum_{j\in\mathbb Z} a_{n-jN}\xi_j.
\]

The following result reflects the scaling law for the father function $\varphi_0$ of the wavelet under consideration,
\begin{equation}
\label{scaling}
\varphi_0(x)=\sqrt{N}\sum_{k\in\mathbb Z}a_k\varphi_0(Nx-k),\quad x\in\mathbb R,
\end{equation}
where
\begin{equation}
\label{scaling2}
m_0(x)=\sum_{k\in\mathbb Z}a_ke^{2\pi  i kx}.
\end{equation}

\begin{lemma}
For the operators $K$ and $S$ (see $(1)$ and $(2)$ in Table \ref{rubicon2} above) we define
\[
K\,:\,\ell_2(\mathbb Z)\,\longrightarrow\,\mathcal H_0 \,\,(\text{the zero resolution subspace in $\mathbf L_2(\mathbb R)$})
\]
by
\[
(K\xi)(x)=\sum_{n\in\mathbb Z}\xi_n\varphi_0(x-n).
\]
Then 
\begin{equation}
\label{KUUK}
KS=UK
\end{equation}
holds, that is the following diagram is commutative:
\[
\begin{array}{ccc}
  \mathbf L_2(\mathbb R)&\xrightarrow{\hspace*{0.5cm}{U}\hspace*{0.5cm}}&\mathbf L_2(\mathbb R) \\
\hspace*{-0.5cm}
\uparrow{\hspace*{0.5cm}\hspace*{-1.2cm}K}
& &\hspace*{0.5cm}\uparrow{K}\\
\ell_2(\mathbb Z)&\xrightarrow{\hspace*{0.5cm}{S}\hspace*{.5cm}}&\ell_2(\mathbb Z),
\end{array}\]
where 
\[
(U\gamma)(x)=\frac{1}{\sqrt{N}}\gamma(x/N),\quad \gamma\in\mathbf L_2(\mathbb R,dx).
\]

\end{lemma}
\begin{proof} We have for $\xi\in\ell_2(\mathbb Z)$:
\[
(KS\xi)(x)=\sum_{n\in\mathbb Z}\sum_{j\in\mathbb Z}a_{n-jN}\xi_j\varphi_0(x-n)
\]
and
\[
\begin{split}
(UK\xi)(x)&=\sum_{j\in\mathbb Z}\xi_j\frac{1}{\sqrt{N}}\varphi_0\left(\frac{x-jN}{N}\right)\\
&=\sum_{j\in\mathbb Z}\sum_{k\in\mathbb Z}\xi_ja_k\varphi_0(x-jN-k)\\
&\text{and, with the change of variable $n=jN+k$,}\\
&=\sum_{n\in\mathbb Z}\sum_{j\in\mathbb Z}\xi_ja_{n-jN}\varphi_0(x-n),
\end{split}
\]
and the result follows.
\end{proof}

More generally for many  choices of filters (see \eqref{scaling2}) there are no $\mathbf L_2(\mathbb R,dx)$-solution to \eqref{scaling}, and
then  one leaves the setting of $\mathbf L_2(\mathbb R)$. We still get
counterparts of \eqref{jardin-du-luxembourg} and \eqref{eqnorm12341234} using the solenoid.\\ 

\begin{proposition}
The operator $R$ in \eqref{ruelle1234} is bounded from $\mathbf L_2(\mathbb T,d\lambda)$ into itself, and its adjoint is given by the formula
\begin{equation}
(R^*f)(z)=|m_0(z)|^2f(z^N).
\label{adjointR}
\end{equation}
\end{proposition}

We now discuss the multiresolution associated with $m_0$ and its relationships with the multiresolution ${\rm MR}_\sigma$. We first note that 
the space $\mathcal H_0$ defined by \eqref{llg1973} is equal to the closed linear span of the functions $x\mapsto \widehat{\varphi_0(x+k)}$, 
when $k$ runs through $\mathbb Z$. In general the family of functions $x\mapsto \varphi_0(x+k)$ ($k\in\mathbb Z$) is not orthogonal in 
$\mathbf L_2(\mathbb R,dx)$.

\begin{proposition}
Let $W=|m_0|^2$, let
\begin{equation}
\label{190616}
h_{\varphi_0}(t)=\sum_{n\in\mathbb Z}|\widehat{\varphi_0}(t+n)|^2,
\end{equation}
and let
\begin{equation}
(Rf)(t)=\frac{1}{N}\left(\sum_{k=0}^{N-1}(Wf)\left(\frac{t+k}{N}\right)\right).
\end{equation}
Then 
\begin{equation}
\label{445}
Rh_{\varphi_0}=h_{\varphi_0}.
\end{equation}
\label{prop49}
\end{proposition}
\begin{proof}
We have
\[
\begin{split}
\left(Rh_{\varphi_0}\right)(t)&=\frac{1}{N}\sum_{k\in \mathbb Z_N}W\left(\frac{t+k}{N}\right)\sum_{n\in\mathbb Z}\big|\widehat{\varphi_0}\left(\frac{
t+k+nN}{N}\right)\big|^2\\
&=\sum_{k\in \mathbb Z_N}\sum_{n\in\mathbb Z}\big|\widehat{\varphi_0}\left(t+\overbrace{k+nN}^{m}\right)\big|^2\\
&=\sum_{m\in\mathbb Z}|\widehat{\varphi_0}(t+m)\big|^2\\
&=h_{\varphi_0}(t),
\end{split}
\]
where we wrote $\mathbb Z_N$ for the cyclic group $\mathbb Z/N\mathbb Z$, and we used the Euclidean algorithm on $\mathbb Z$, mod $N$, in the 
last step ($m=k+nN$). The first step used the  scaling identity for $\varphi_0$ and $W=|m_0|^2$.
\end{proof}

As an application of Proposition \ref{prop49} we get the following results for wavelets on solenoids.

\begin{corol}
Let $W=|m_0|^2$ and $h_{\varphi_0}$ be as in Proposition \ref{prop49}; then $M_n\stackrel{\rm def.}{=}h_{\varphi_0}(\pi_n)$ is a 
$(\pi_n)_{n\in\mathbb N_0}$-martingale, where $(\pi_n)_{n\in\mathbb N_0}$ denotes the ${\rm Sol}_N(\mathbb T)$-Markov chain.
\end{corol}

\begin{corol}
Consider the wavelet filter $m_0$ with scaling function $\varphi_0\in\mathbf L_2(\mathbb R)$. Let $h_{\varphi_0}$ be the corresponding
harmonic function: $R_{m_0}h_{\varphi_0}=h_{\varphi_0}$, see \eqref{445}. Then the level-$0$ isometry
\[
V_0\,:\,f\in\mathbf L_2(\mathbb T,h_{\varphi_0}(t)dt)\,\,\mapsto\,\, \left(f(t)\widehat{\varphi_0}(t)\right)\in\mathbf L_2(\mathbb R)
\]
has the following explicit adjoint $V_0^*$ computed on $\mathbf L_2(\mathbb R)$:
\[
\left(V_0^*\gamma\right)(t)=\frac{1}{h_{\varphi_0}(t)}\sum_{n\in\mathbb Z}\gamma(t+n)\overline{\widehat{\varphi_0}(t+n)},\quad t\in\mathbb R,\quad
\gamma\in\mathbf L_2(\mathbb R,dx).
\]
\end{corol}

\begin{remark}{\rm For functions $k$ defined on $[0,1]$ (or, equivalently, on $\mathbb R/\mathbb Z$) we introduce the Fourier
coefficients
\[
\widehat{k}(n)=\int_0^1e^{-2\pi int}k(t)dt,\quad n\in\mathbb Z.
\]
With $k=h_{\varphi_0}$ from \eqref{190616} we can then compute the inner products $\int_{\mathbb R}\varphi_0(x+n)\overline{\varphi_0(x)}dx$ for
$n\in\mathbb Z$. See the following proposition.}
\end{remark}

\begin{proposition} Let $h_{\varphi_0}$ be the harmonic function associated to a scaling function $\varphi_0\in\mathbf L_2(\mathbb R,dx)$. 
Then the following hold:\\
$(i)$
\[
\int_{\mathbb R}\varphi_0(x+n)\overline{\varphi_0}(x)dx=\widehat{h_{\varphi_0}}(n),\quad n\in\mathbb Z.
\]
$(ii)$ The generating function
\[
\zeta\in\mathbb C\,\mapsto\,G_{\varphi_0}(\zeta)=\sum_{n\in\mathbb Z}\zeta^n\left(\int_{\mathbb R}\varphi_0(x+n)\overline{\varphi_0}(x)dx\right)
\]
is an analytic extension of $h_{\varphi_0}$ to an open neighborhood of $\mathbb T$.\\
$(iii)$ The scaling function $\varphi_0$ is compactly supported on $\mathbb R$  if and only if $G_{\varphi_0}$ is a polynomial.
\end{proposition}
\begin{proof} We need only to prove $(i)$. The other two claims follow then easily. Using Parseval's equality in $\mathbf L_2(\mathbb R)$
we have
\[
\begin{split}
\int_{\mathbb R}\varphi_0(x+n)\overline{\varphi_0}(x)dx&=\int_{\mathbb R}e^{-2\pi i nt}|\widehat{\varphi_0}(t)|^2dt\\
&=\sum_{m\in\mathbb Z}\int_0^1e^{-2\pi i nt}|\widehat{\varphi_0}(t+m)|^2dt\\
&\text{and using the dominated convergence theorem,}\\
&=\int_0^1e^{-2\pi i nt}\left(\sum_{m\in\mathbb Z}|\widehat{\varphi_0}(t+m)|^2\right)dt\\
&=\int_0^1e^{-2\pi i nt}h_{\varphi_0}(t)dt\\
&=\widehat{h_{\varphi_0}}(n),\quad \forall n\in\mathbb Z.
\end{split}
\]
\end{proof}

\begin{corol} Orthogonality of the family $\left\{\varphi_0(\cdot+n)\right\}_{n\in\mathbb Z}$ in $\mathbf L_2(\mathbb R,dx)$ is  
equivalent to the condition $h_{\varphi_0}\equiv 1$.
\end{corol}

In the next example we show that the Fej\'er kernels arise as $h_{\varphi_0}$ for a family of scaling
functions $\varphi_0\in\mathbf L_2(\mathbb R,dx)$.
We first recall that the Dirichlet kernel and Fej\'er kernels are defined respectively by
\[
D_k(\zeta)=\sum_{j=-k}^k\zeta^j
\]
and
\[
F_k(\zeta)=\frac{\sum_{u=0}^k D_u(\zeta)}{k+1}.
\]

\begin{example}
We take $\varphi_0(x)=\frac{1}{\sqrt{2m+1}}\chi_{[0,2m+1]}(x)$, where $m\in\mathbb N$ is  fixed. Then
\[
\int_{\mathbb R}\varphi_0(x)\varphi_0(x-n)dx=\begin{cases}\,\, 0,\,\,\,\hspace{0.9cm}{\rm if}\,\,\, |n|\ge 2m+1,\\
\, \frac{2m+1-n}{2m+1}\,\,\,\,{\rm for} \,\,\, n\in\left\{0,\ldots, 2m\right\}.
\end{cases}
\]
Thus
\[
(2m+1)h_{\varphi_0}(\zeta)=\zeta^{-2m}+2\zeta^{1-2m}+\cdots +(2m)\zeta^{-1}+(2m+1)+(2m)\zeta+\cdots +2\zeta^{2m-1}+\zeta^{2m},
\]
which is the Fej\'er kernel $F_{2m}$.
\end{example}

\subsubsection{Realization using the solenoid}

We set  $e_n(z)=z^n$, $n\in\mathbb Z$.

\begin{theorem}
Let $z\in\mathbb T$. The function
\begin{equation}
L\left(\frac{n}{N^k}\right)=\left(R^k\left(e_nh\right)\right)(z)
\end{equation}
is positive definite on $\mathbb Z[1/N]$, and there exists a positive finite measure $d\mu_z$ on ${\rm Sol}_N(\mathbb T)$ such that
\begin{equation}
L\left(\frac{n}{N^k}\right)=\int_{{\rm Sol}_N(\mathbb T)}\chi\left(\frac{n}{N^k}\right)(x)dP_z(x)
\end{equation}
\end{theorem}

{\bf Proof:} We first check that $L$ is well defined. We have
\[
\begin{split}
L\left(\frac{Nn}{N^{k+1}}\right)&=\left(R^{k+1}\left(e_{Nn}h\right)\right)(z)\\
&=\left(R^k\left(Re_{nN}h\right)\right)(z)\\
&=\left(R^k\left(\sum_{\substack{w\in\mathbb T\\ w^N=z}}|m_0(w)|^2e_{nN}(w)h(w)\right)\right)(z)\\
&=\left(R^k\left(\sum_{\substack{w\in\mathbb T\\ w^N=z}}|m_0(w)|^2e_{n}(z)h(w)\right)\right)(z)\quad\text{(since $e_{nN}(w)=e_n(z)$)}\\
&=\left(R^k\left(e_n\sum_{\substack{w\in\mathbb T\\ w^N=z}}|m_0(w)|^2h(w)\right)\right)(z)\\
&=\left(R^ke_nRh\right)(z)\\
&=\left(R^ke_nh\right)(z)\\
&=L\left(\frac{n}{N^{k}}\right).
\end{split}
\]
We now prove that $L$ is positive definite on $\mathbb Z[1/N]$. Let $M\in\mathbb N$, $c_1,\ldots, c_M\in\mathbb C$ and
$\frac{n_1}{N^{k_1}},\ldots, \frac{n_M}{N^{k_M}}\in\mathbb Z[1/N]$. In view of the first part of the proof, we assume all the denominators equal, say to $k$. We have
\[
\begin{split}
\sum_{u,v=1}^M \overline{c_u}c_vL\left(\frac{n_u}{N^{k}}-\frac{n_v}{N^{k}}\right)&=\sum_{u,v=1}^M \overline{c_u}c_vR^k\left((e_{n_u-n_v})h\right)(z)\\
&=\sum_{u,v=1}^M \overline{c_u}c_vR^k\left((e_{n_u-n_v})h\right)(z)\\
&=\sum_{u,v=1}^M \overline{c_u}c_vR^k\left((e_{n_u}\overline{e_{n_v}})h\right)(z)\\
&=\left(R\left(|g|^2h\right)\right)(z)\ge 0,
\end{split}
\]
with $g=\sum_{u=1}^M\overline{c_u}e_u$.\\

The second claim comes from Bochner's theorem for compact groups.

\mbox{}\qed\mbox{}\\

Let
\[
\pi_0(z)=z_0,
\]
and
\[
U(\psi)=m_0(\pi_0(z))(\psi\circ \sigma_N)(z).
\]
\begin{proposition}
$U$ is unitary and its adjoint is given by the formula
\begin{equation}
U^*\psi=\frac{1}{m\circ\pi_1}\psi\circ\sigma_N^{-1}.
\end{equation}
\end{proposition}
{Proof:}   The results follow from the previous considerations; see also  \cite{MR1837681}.

\mbox{}\qed\mbox{}\\

\subsubsection{Multiresolutions}

We set
\[
\mathcal L_k=\text{\rm closed linear span}\,\left\{\chi\left(\frac{n}{N^k}\right),\, n\in\mathbb Z\right\},\quad k\in\mathbb Z.
\]

\subsubsection{Embedding the real line into the solenoid}

We define
\begin{equation}
\label{gammaN}
\gamma_N(t)=\left(e^{2\pi i[t]},e^{2\pi i[t/N]},e^{2\pi i[t/N^2]},\ldots\right)\in\prod_{n=0}^\infty\left(\mathbb R/\mathbb Z\right),
\end{equation}
where $[x]$ denotes the value of $x\in\mathbb Z$ modulo $1$.
\begin{lemma}
The map $\gamma_N$ is one-to-one from $\mathbb R$ into ${\rm Sol}_N(\mathbb T)$, meaning that
\[
\gamma_N(t)=(1,1,1,\ldots)\,\,\,\iff\,\,\, t=0.
\]
\end{lemma}

{\bf Proof:} See \cite{MR1837681}.

\mbox{}\qed\mbox{}\\

\subsubsection{Probability}

Let $h\ge 0$ be such that $Rh=h$, and assume that $\int_{\mathbb T}h(\lambda)d\lambda=1$. 

\begin{proposition}
The distribution of the random variable
\[
\pi_k(\mathbf z)=z_k
\]
is $|m^{(k)}(z)|^2h(z)d\lambda$, where
\begin{equation}
m^{(k)}(z)=m_0(z)m_0(z^N)\cdots m_0(z^{N^{k-1}}).
\end{equation}
\end{proposition}

{\bf Proof:} We want to show that for every bounded measurable function $f$ on $\mathbb T$ we have
\[
\int_{\mathbb T} f(e^{it})d\lambda(e^{it})=\int_{{\rm Sol}_N(\mathbb T)}f(\pi_k(\mathbf z))dP(\mathbf z).
\]
We prove this equality for $f(z)=z^n$ (which we denoted by $e_n(z)$ ) with $n\in\mathbb Z$, that is
\begin{equation}
\label{Len}
\int_{\mathbb T}e^{int}d\lambda(e^{it})=\int_{{\rm Sol}_N(\mathbb T)}e_n(\pi(\mathbf z))dP(\mathbf z).
\end{equation}
But recall that
\[
\chi\left(\frac{n}{N^k}\right)(\mathbf z)=z_k^n=(\pi_k(\mathbf z))^n=e_n(\pi_k(\mathbf z)).
\]
Hence
\[
\begin{split}
\int_{{\rm Sol}_N(\mathbb T)}\chi\left(\frac{n}{N^k}\right)(\mathbf z)dP(\mathbf z)&=\int_{\mathbb T}R^k(e_nh)(e^{it})d\lambda(e^{it})\\
&=\int_{\mathbb T}(R^{*k}(1))(e^{it})e_n(e^{it})h(e^{it})d\lambda(e^{it})\\
&=\int_{\mathbb T}|m^{(k)}(e^{it})|^2h(e^{it})d\lambda(e^{it}),
\end{split}
\]
where we have used \eqref{adjointR} to compute $R^{*k}1$.
\mbox{}\qed\mbox{}\\

\begin{remark}
{\rm
We note that
\[
\begin{split}
\int_{\mathbb T}|m^{(k)}(e^{it})|^2h(e^{it})d\lambda(e^{it})&=\int_{\mathbb T}(R^{*k}(1))(e^{it})h(e^{it})d\lambda(e^{it})\\
&=\int_{\mathbb T}(R^{k}(h))(e^{it})d\lambda(e^{it})\\
&=\int_{\mathbb T}h(e^{it})d\lambda(e^{it})=1,
\end{split}
\]
as it should be.
}
\end{remark}

\subsubsection{The martingale property}

As in the previous section $h$ denotes a positive function such that $Rh=h$ and $\int_{\mathbb T}h(\lambda)d\lambda=1$. 
Let $z,w\in\mathbb T$ be such that $w^N=z$, and set
\[
P(z\mapsto w)=\frac{1}{N}|m(w)|^2h(w).
\]
The Markov property now reads
\begin{equation}
\label{eqmarkov}
\sum_{\substack{w\in\mathbb T\\ w^N=z}}P(z\mapsto w)=1.
\end{equation}

The martingale property is now
\[
R(\xi_{n+1}h)=\xi_nh,
\]
and the following formulas hold:
\[
\begin{split}
\frac{R(fh)}{h}&\le 1\quad (\text{conditional expectation})\\
\frac{R(\xi(z^n))h}{h}&=\xi\\
\sum_{w^N=z}\xi(w)h(w)&=\xi(z)\frac{Rh}{h}\\
M_{z_0z_1}M_{z_1z_2}\cdots &=R^k,
\end{split}
\]
with
\[
M_{z_0z_1}M_{z_1z_2}=\frac{1}{N}\sum_{z_1^n=z_0}|m(z_1)|^2\sum_{z_2^N=z_1}|m(z_1)|^2.
\]

\subsection{Fractal examples}
We here consider $X$ to be the set of numbers of the form
\[
x=\sum_{n=1}^\infty\frac{b_n}{3^n},\quad {\rm where}\quad b_n\in\left\{0,2\right\}
\]
and $\sigma(x)=3x$ (mod $1$). In symbolic form we have
\[
(b_1,b_2,b_3,\ldots)\,\xrightarrow{\hspace{.3cm}{\sigma}\hspace*{.3cm}}\,(b_2,b_3,b_4,\ldots)
\]
More generally, let $d\in\mathbb N$ and let $A\in\mathbb Z^{d\times d}$ with all eigenvalues of modulus strictly bigger than $1$, and let 
$m<|\det A|$. Fix $d$ residue classes $b_1,\ldots, b_d$ 
in $\mathbb Z^d/A\mathbb Z^d$. We set $Y=\left\{b_1,\ldots, b_n\right\}$.
We consider the set $X$ of vectors in $\mathbb R^d$ of the form
\[
x=\sum_{n=1}^\infty A^{-n}c_n
\]
where to make connections with homogeneous Markov chains (see \eqref{agathe} for the latter) we define

\begin{equation}
F_{b_j}(x)=A^{-1}(x+b_j),\quad j=1,\ldots, d
\end{equation}
and $\sigma(x)=Ax$ modulo $\mathbb Z^d$.\smallskip

Since $Y$ is a finite set, a probability measure on $Y$ is given by a finite number of positive numbers $p_1,\ldots ,p_d$ adding up to $1$, and the transfer operator
is now given by
\begin{equation}
Rf=\sum_{j=1}^d p_jf\circ F_{b_j}.
\end{equation}

Consider the set ${\rm Prob}(X)$ of probabilities on $X$, and let $\nu,\mu\in {\rm Prob}(X)$. Consider 
the Hausdorff distance between $\nu$ and $\mu$:
\[
d_H(\nu,\mu)=\sup\left\{ \int_Xf(x)(d\nu(x)-d\mu(x))\right\}
\]
where the supremum is on the set of all Lipschitz functions:
\[
|f(x)-f(y)|\le \|x-y\|,\quad(\text{with $\|x-y\|$ being the usual distance in $\mathbb R^d$}).
\]
Define a measure $\nu R$ on $X$ via
\[
\int_X f(d(\nu R))=\int_X(Rf)d\nu.
\]
A theorem of Hutchinson (see \cite{BrJo02a,MR625600}) states that the map  $\nu\mapsto \nu R$ is then strictly contractive. 
There exists $\alpha\in(0,1)$ such that
\[
d_H(R\nu,R\mu)\le\alpha d_H(\nu,\mu).
\]
Existence and uniqueness of a solution to the equation $\nu R=\nu$ follows from Banach fixed point theorem.
\subsection{The Gauss operator}
The present example is related to number theory and has links with information theory; see \cite{blachman1965,MR755797,MR1176073,Ryll1951f}.
We take $X=(0,1)$ and $d\lambda(x)=dx$, and
\begin{equation}
\sigma(x)=<\frac{1}{x}>,
\end{equation}
where $<\cdot>$ denotes the ``fractional part'', defined as follows: If $x\in(\frac{1}{k+1},\frac{1}{k})$ then $\sigma(x)=\frac{1}{x}-k$. 
We also define $\tau_k(x)=\frac{1}{x+k}$ with $k=1,2,\ldots$. Note that
\[
\sigma\circ\tau_k(x)=x,\quad k=1,2,\ldots
\]
The solenoid (see Definition \ref{defsol123}) associated with $X$ is described as follows:
\begin{equation}
\label{sol-voltaire}
{\rm Sol}_\sigma(0,1)=\left\{(x_0,x_1,\ldots )\in\prod_{n=0}^\infty (0,1)\,\,\mbox{\text{such that}}\,\, x_k=<\frac{1}{x_{k+1}}>\right\}.
\end{equation}
We thus obtain the continued fraction associated with $x_0$.
\[
\begin{split}
x_0&=k_1+\dfrac{1}{\sigma(x_1)}\\
&=k_1+\dfrac{1}{k_2+\dfrac{1}{\sigma(x_2)}}\\
&=k_1+\dfrac{1}{k_2+\dfrac{1}{k_3+\dfrac{1}{\ddots}}}\\
&\hspace{2mm}\vdots
\end{split}
\]

Now the transfer operator is given by:
\begin{equation}
(Rf)(x)=\sum_{n=1}^\infty\frac{1}{(n+x)^2}f\left(\frac{1}{n+x}\right).
\end{equation}
\begin{proposition}
Let $h(x)=\frac{1}{\ln2}\frac{1}{1+x}$. Then $\int_0^1h(x)dx=1$ and $\lambda R=\lambda$.
\end{proposition}
\begin{proof}
We have
\[
\begin{split}
\lambda R=\lambda&\iff \int_0^1\left(\sum_{n=1}^\infty \frac{1}{(n+x)^2}f\left(\frac{1}{n+x}\right)\right)dx=\int_0^1f(x)dx\\
&\iff \sum_{n=1}^\infty\int_0^1\frac{1}{(n+x)^2}f\left(\frac{1}{n+x}\right)dx
= \int_0^1f(x).
\end{split}
\]
Note that the change of variable $y=\frac{1}{n+x}$ leads to
\[
\int_{0}^1\frac{1}{(n+x)^2}f\left(\frac{1}{n+x}\right)dx=\int_{\frac{1}{n+1}}^{\frac{1}{n}}f(y)dy,
\]
and hence
\[
\sum_{n=1}^\infty\int_0^1\frac{1}{(n+x)^2}f\left(\frac{1}{n+x}\right)dx= \sum_{n=1}^\infty\int_{\frac{1}{n+1}}^{\frac{1}{n}}f(x)dx=\int_0^1f(x)dx,
\]
and the result follows.
\end{proof}

\bibliographystyle{plain}
\def\cfgrv#1{\ifmmode\setbox7\hbox{$\accent"5E#1$}\else
  \setbox7\hbox{\accent"5E#1}\penalty 10000\relax\fi\raise 1\ht7
  \hbox{\lower1.05ex\hbox to 1\wd7{\hss\accent"12\hss}}\penalty 10000
  \hskip-1\wd7\penalty 10000\box7} \def\cprime{$'$} \def\cprime{$'$}
  \def\cprime{$'$} \def\lfhook#1{\setbox0=\hbox{#1}{\ooalign{\hidewidth
  \lower1.5ex\hbox{'}\hidewidth\crcr\unhbox0}}} \def\cprime{$'$}
  \def\cprime{$'$} \def\cprime{$'$} \def\cprime{$'$} \def\cprime{$'$}
  \def\cprime{$'$}

\end{document}